\documentclass[10pt]{amsart}
\usepackage{verbatim}
\usepackage{amssymb}
\usepackage{upref}
\usepackage[all]{xy}
\usepackage[usenames,dvipsnames]{color}
\usepackage{url}
\usepackage[normalem]{ulem}

\allowdisplaybreaks

\newtheorem{thm}{Theorem}[section]
\newtheorem*{thm*}{Theorem}
\newtheorem{lem}[thm]{Lemma}
\newtheorem{cor}[thm]{Corollary}
\newtheorem{prop}[thm]{Proposition}

\newtheorem{lemma}[thm]{Lemma}

\theoremstyle{definition}
\newtheorem{defn}[thm]{Definition}
\newtheorem{definition}[thm]{Definition}

\newtheorem{example}[thm]{Example}

\newtheorem*{notn*}{Notation}

\newtheorem*{hyp*}{Hypothesis}

\newtheorem{rem}[thm]{Remark}
\newtheorem{remark}[thm]{Remark}
\newtheorem*{rem*}{Remark}

\numberwithin{equation}{section}

\newcommand{\thmref}[1]{Theorem~\textup{\ref{#1}}}
\newcommand{\corref}[1]{Corollary~\textup{\ref{#1}}}
\newcommand{\lemref}[1]{Lemma~\textup{\ref{#1}}}
\newcommand{\propref}[1]{Proposition~\textup{\ref{#1}}}

\newcommand{\defnref}[1]{Definition~\textup{\ref{#1}}}

\newcommand{\midtext}[1]{\quad\text{#1}\quad}
\newcommand{\righttext}[1]{\quad\text{#1 }}
\renewcommand{\and}{\midtext{and}}

\renewcommand{\)}{\textup)}

\newcommand{\EE}{\mathcal E}

\newcommand{\BB}{\mathcal B}

\renewcommand{\AA}{\mathcal A}

\renewcommand{\epsilon}{\varepsilon}

\newcommand{\id}{\text{\textup{id}}}

\newcommand{\<}{\langle}
\renewcommand{\>}{\rangle}

\newcommand{\under}{\backslash}
\newcommand{\inv}{^{-1}}

\renewcommand{\bar}{\overline}
\newcommand{\what}{\widehat}

\newcommand{\ibm}{imprimitivity bimodule}

\usepackage{tikz}

\usetikzlibrary{cd,decorations.pathmorphing}
\usepackage{mathtools}
\mathtoolsset{showonlyrefs} 
\usepackage[alphabetic]{amsrefs}

\definecolor{refkey}{cmyk}{0.93,0.33,0.92,0.25}
\definecolor{labelkey}{cmyk}{0.93,0.33,0.92,0.25}
\definecolor{Dgreen}{cmyk}{0.93,0.33,0.92,0.25} 

\newcommand\set[1]{\{\,#1\,\}}
\newcommand\go{G^{(0)}}
\newcommand\ho{H^{(0)}}

\newcommand\guh{G\backslash H}
\newcommand\gua{G\backslash\AA}
\newcommand\lambdah{\lambda_{H}}
\newcommand\lambdag{\lambda_{G}}
\newcommand\lambdaguh{\lambda_{\guh}}
\newcommand\tensor\otimes

\newcommand\pb{\bar p}

\newcommand\sd[2]{S(#1,#2)}
\newcommand\ac[2]{#1\sdp #2}

\newcommand\sdp{\rtimes}
\newcommand\hsdg{\sd HG}

\def\tg(#1,#2){\tau_{G}(#1,#2)}
\def\th(#1,#2){\tau_{H}(#1,#2)}

\newcommand\jqphi{\bar\rho'}
\newcommand\jqpsi{\bar\sigma}
\newcommand\Aut{\operatorname{Aut}}

\let\ipscriptstyle=\scriptscriptstyle
\def\lipsqueeze{{\mskip -3.0mu}}
\def\ripsqueeze{{\mskip -3.0mu}}
\def\ipcomma{\nobreak\mathrel{,}\nobreak}
\newbox\ipstrutbox
\setbox\ipstrutbox=\hbox{\vrule height8.5pt
width 0pt}
\def\ipstrut{\copy\ipstrutbox}
\def\lip#1<#2,#3>{\mathopen{\relax_{\ipstrut\ipscriptstyle{
#1}}\lipsqueeze
\langle} #2\ipcomma #3 \rangle}
\def\blip#1<#2,#3>{\mathopen{\relax_{\ipstrut
\ipscriptstyle{ #1}}\lipsqueeze\bigl\langle} #2\ipcomma #3 \bigr\rangle}
\def\rip#1<#2,#3>{\langle #2\ipcomma #3
\rangle_{\ripsqueeze\ipstrut\ipscriptstyle{#1}}}
\def\brip#1<#2,#3>{\bigl\langle #2\ipcomma #3
\bigr\rangle_{\ripsqueeze\ipstrut\ipscriptstyle{#1}}}
\def\angsqueeze{\mskip -6mu}
\def\smangsqueeze{\mskip -3.7mu}
\def\trip#1<#2,#3>{\langle\smangsqueeze\langle #2\ipcomma #3
\rangle\smangsqueeze\rangle_{\ripsqueeze\ipstrut\ipscriptstyle{#1}}}
\def\btrip#1<#2,#3>{\bigl\langle\angsqueeze\bigl\langle #2\ipcomma
#3
\bigr\rangle
\angsqueeze\bigr\rangle_{\ripsqueeze\ipstrut\ipscriptstyle{#1}}}
\def\tlip#1<#2,#3>{\mathopen{\relax_{\ipstrut\ipscriptstyle{
#1}}\lipsqueeze \langle\smangsqueeze\langle} #2\ipcomma #3
\rangle\smangsqueeze\rangle}
\def\btlip#1<#2,#3>{\mathopen{\relax_{\ipstrut\ipscriptstyle{
#1}}\lipsqueeze
\bigl\langle\angsqueeze\bigl\langle} #2\ipcomma #3
\bigr\rangle\angsqueeze\bigr\rangle}

\def\ip(#1|#2){(#1\mid #2)}
\def\bip(#1|#2){\bigl(#1 \mid #2\bigr)}
\def\Bip(#1|#2){\Bigl( #1 \bigm| #2 \Bigr)}
\def\Lip<#1,#2>{\lip\scriptstyle\star<#1,#2>}
\def\Rip<#1,#2>{\rip\scriptstyle\star<#1,#2>}

\newcommand\cs{\ensuremath{C^{*}}}
\newcommand{\ib}{im\-prim\-i\-tiv\-ity bi\-mod\-u\-le}

\newcommand{\sme}{\,\mathord{\mathop{\text{--}}\nolimits_{\relax}}\,}

\usepackage[shortlabels]{enumitem}
\newcommand{\pre}{{}_}
\newcommand{\actleft}{\curvearrowright}
\newcommand{\actright}{\curvearrowleft}

\newcommand{\lchs}{locally compact Hausdorff space}
\newcommand{\lcg}{locally compact groupoid}
\definecolor{amethyst}{rgb}{0.6,0.4.0.8}
\definecolor{airforceblue}{rgb}{0.36, 0.54, 0.66}
\definecolor{alizarin}{rgb}{0.82, 0.1, 0.26}
\definecolor{amaranth}{rgb}{0.9, 0.17, 0.31}
\definecolor{amber}{rgb}{1.0, 0.75, 0.0}
\definecolor{amber(sae/ece)}{rgb}{1.0, 0.49, 0.0}
\definecolor{americanrose}{rgb}{1.0, 0.01, 0.24}
\definecolor{antiquefuchsia}{rgb}{0.57, 0.36, 0.51}
\definecolor{ao}{rgb}{0.0, 0.0, 1.0}
\definecolor{ao(english)}{rgb}{0.0, 0.5, 0.0}
\definecolor{applegreen}{rgb}{0.55, 0.71, 0.0}
\definecolor{aquamarine}{rgb}{0.5, 1.0, 0.83}
\definecolor{arylideyellow}{rgb}{0.91, 0.84, 0.42}
\definecolor{asparagus}{rgb}{0.53, 0.66, 0.42}
\definecolor{atomictangerine}{rgb}{1.0, 0.6, 0.4}
\definecolor{aureolin}{rgb}{0.99, 0.93, 0.0}
\definecolor{awesome}{rgb}{1.0, 0.13, 0.32}
\definecolor{azure(colorwheel)}{rgb}{0.0, 0.5, 1.0}
\definecolor{ballblue}{rgb}{0.13, 0.67, 0.8}
\definecolor{bananayellow}{rgb}{1.0, 0.88, 0.21}
\definecolor{bittersweet}{rgb}{1.0, 0.44, 0.37}
\definecolor{bleudefrance}{rgb}{0.19, 0.55, 0.91}
\definecolor{blue-violet}{rgb}{0.54, 0.17, 0.89}
\definecolor{blush}{rgb}{0.87, 0.36, 0.51}
\definecolor{brass}{rgb}{0.71, 0.65, 0.26}
\definecolor{brickred}{rgb}{0.8, 0.25, 0.33}
\definecolor{brightgreen}{rgb}{0.4, 1.0, 0.0}
\definecolor{brightlavender}{rgb}{0.75, 0.58, 0.89}
\definecolor{brightmaroon}{rgb}{0.76, 0.13, 0.28}
\definecolor{brightpink}{rgb}{1.0, 0.0, 0.5}
\definecolor{brightturquoise}{rgb}{0.03, 0.91, 0.87}

\newcommand{\units}{^{(0)}}
\newcommand{\pairs}{^{(2)}}

\begin{document}
\title[Groupoid semidirect product Fell bundles II]{Groupoid semidirect product Fell bundles II ---\\
Principal actions and stabilization}

\author[Hall]{Lucas Hall}
\address{School of Mathematical and Statistical Sciences
\\Arizona State University
\\Tempe, Arizona 85287}
\email{lhall10@asu.edu}

\author[Kaliszewski]{S. Kaliszewski}
\address{School of Mathematical and Statistical Sciences
\\Arizona State University
\\Tempe, Arizona 85287}
\email{kaliszewski@asu.edu}

\author[Quigg]{John Quigg}
\address{School of Mathematical and Statistical Sciences
\\Arizona State University
\\Tempe, Arizona 85287}
\email{quigg@asu.edu}

\author[Williams]{Dana P. Williams}
\address{Department of Mathematics
\\Dartmouth College
\\Hanover, New Hampshire 03755}
\email{dana@math.dartmouth.edu}



\date{May 4, 2021}

\begin{abstract}
  Given a free and proper action of a groupoid on a Fell bundle (over
  another groupoid), we give an equivalence between the
  semidirect-product and the generalized-fixed-point Fell bundles,
  generalizing an earlier result where the action was by a group.  As
  an application, we show that the 
  Stabilization Theorem for Fell bundles over groupoids
  is essentially another form of crossed-product duality.
\end{abstract}
\maketitle

\section{Introduction}

Let $H$ be a locally compact Hausdorff groupoid, and let $p:\AA\to H$
be a Fell bundle.  \cite[Corollary~3.4]{kmqw2} shows that if a locally
compact group $G$ acts principally (i.e., freely and properly) by
automorphisms on $\AA$, then there is a ``Yamagami equivalence''
between the semidirect-product Fell bundle
$\AA\rtimes G\to H\rtimes G$ and the quotient Fell bundle
$G\under \AA\to G\under H$.  The equivalence theorem of
\cite[Theorem~6.4]{mw:fell} (based upon an unpublished preprint of
Yamagami \cite{yamagami}) then gives a Morita equivalence
$C^*(H\rtimes G,\AA\rtimes G)\sim_M C^*(G\under H,G\under \AA)$
between the Fell-bundle algebras.  \cite[Theorem~3.1]{kmqw3} shows
that this should be regarded as a version of Rieffel's equivalence
theorem \cite[Corollary~1.7]{proper} for a ``universal'' generalized
fixed-point algebra, whereas Rieffel's original theorem is in some
sense a ``reduced'' version (see also \cite[Proposition~2.2]{BusEch}).
Similarly, $C^*(G\under H,G\under \AA)$ should be regarded the
universal generalized fixed-point algebra for the action of $G$ on
$\AA$.

In this paper we generalize \cite[Corollary~3.4]{kmqw2} to allow $G$
to be a locally compact groupoid.  This requires some preliminary
work: a key tool in \cite{kmqw2} is a result of Palais (see the
discussion preceding \cite[Proposition~1.3.4]{palais}) that
characterizes principal group bundles as pullbacks. This was parlayed
in \cite[Theorem~A.16]{kmqw2} to a characterization of Fell bundles
carrying a principal action of a group.  Here we need to begin with a
generalization (\thmref{palais}) of Palais' theorem for principal
bundles associated to an
action of a groupoid, rather than just a group.  We also need to
generalize much of the ``infrastructure'' developed in
\cite[Appendix]{kmqw2} from groups to groupoids.  This includes, for
example, quotients and semidirect products of Fell bundles by actions
of a groupoid, and a Palais-type classification
(\thmref{classify Fell})
of
principal actions of groupoids on Fell bundles.

\section{Preliminaries}

We refer to \cite{hkqwsemi} for our conventions regarding
groupoids,
Fell bundles,
Haar systems,
actions of groupoids by isomorphisms on either
groupoids or Fell bundles,
and
semidirect product groupoids and Fell bundles.

Throughout, $G$ and $H$ will be second countable, locally compact Hausdorff
groupoids with Haar systems $\lambdag=\set{\lambdag^{u}}_{u\in \go}$
and $\lambdah=\set{\lambdah^{v}}_{v\in\ho}$, respectively.

If $T$ is a left $G$-space, then we use $G\backslash T$ to denote the
space of orbits with the quotient topology. Since we assume $G$ has a Haar system, the quotient
map $q:T\to G\backslash T$ is open
\cite{wil:toolkit}*{Proposition~2.12}.

If $p:\AA\to G$ is a Fell bundle and $q:\EE\to T$ is an
upper-semicontinuous Banach bundle over a $G$-space $T$, then as in
\cite{mw:fell}*{\S6},\footnote{Note that there is a typo in
  condition~(c) in \cite{mw:fell}*{\S6}.}  we say that $\AA$ acts on (the left of) $\EE$ if there is a continuous map
$(a,e)\mapsto a\cdot e$ from $\AA*\EE=\set{(a,e): s(a)=q(e)}$ to $\EE$
that is bilinear on $A(\gamma)\times E(s(\gamma))$ and such that
\begin{enumerate}[({A}1)]
\item $q(a\cdot e)=p(a)\cdot q(e)$,
\item $a\cdot (b\cdot e)=(ab)\cdot e$ if $(b,e)\in\AA*\EE$ and
  $(a,b)\in \AA\pairs$, and
\item $\|a\cdot e\|\le \|a\|\|e\|$.
\end{enumerate}

\subsubsection*{Fell Bundle Equivalence}

Suppose that $T$ is a $(G,H)$-equivalence.
Then there is a
continuous open map $(t,t')\mapsto \tg(t,t')$ from $T*_{s}T\to G$ such
that $\tg(t,t')\cdot t'=t$ (see \cite{wil:toolkit}*{Lemma~2.42}).
Similarly, there is a continuous open map $\tau_{H}:T*_{r}T\to H$ such that
$t\cdot \th(t,t')=t'$.
\begin{definition}[\cite{mw:fell}*{Definition~6.1}]
\label{def:fell-bundle-equiv}
  Suppose that $T$ is a $(G,H)$-equivalence and $p:\AA\to G$ and
  $\pb:\BB\to H$ are Fell bundles.   We say that an upper
  semicontinuous Banach bundle $q:\EE\to T$ is an
  $\AA-\BB$-equivalence if the following conditions are met.
  \begin{enumerate}[\rm ({E}1)]
  \item\label{E1} There is a left action of $\AA$ and a right action of $\BB$ on
    $\EE$ such that $a\cdot (e\cdot b)=(a\cdot e)\cdot b$ for all
    $a\in\AA$, $e\in \EE$, and $b\in \BB$.
  \item\label{E2} There are sesquilinear maps $(e,f)\mapsto \Lip<e,f>$ from
    $\EE*_{s}\EE$ to $\AA$ and $(e,f)\mapsto \Rip<e,f>$ from
    $\EE*_{r}\EE$ to $\BB$ such that
    \begin{enumerate}[(a)]
    \item $p(\Lip<e,f>))=\tg(q(e),q(f))$ and
      $\pb(\Rip<e,f>)=\th(q(e),q(f))$, 
    \item $\Lip<e,f>^{*}=\Lip<f,e>$ and $\Rip<e,f>^{*}=\Rip<f,e>$,
    \item $\Lip<a\cdot e,f>=a\Lip<e,f>$ and $\Rip<e,f\cdot
      b>=\Rip<e,f>b$, and
    \item $\Lip<e,f>\cdot g=e\cdot\Rip<f,g>$.
    \end{enumerate}
\item With the actions coming from \ref{E1} and the inner products from
  \ref{E2}, each $E(t)$ is a $A(r(t))\sme B(s(t))$-\ib.
  \end{enumerate}
\end{definition}

The Equivalence Theorem---that is,
\cite{mw:fell}*{Theorem~6.4}---implies that if there is a
$\AA-\BB$-equivalence as above, then $\cs(G,\AA)$ and $\cs(H,\BB)$
are Morita equivalent.

\section{Bundles and Proper Actions}
\label{sec:bundl-prop-acti}

Recall that if $G$ is a groupoid and $T$ is a locally compact $G$-space, then we say that
$G$ acts \emph{properly} or that $T$ is a \emph{proper $G$-space} if
the map $\Theta:G*T\to T\times T$ given by $\Theta(x,t)=(x\cdot t,t)$
is proper in that the inverse image of every compact set in $T\times T$ is
compact in $G*T$.

\begin{prop}
  \label{prop-proper-equiv} Suppose that $T$ is a locally compact
  $G$-space.  Then the following are equivalent.
  \begin{enumerate}[\rm (P{A}1)]
  \item $G$ acts properly on $T$.
   \item For every pair of compact sets $K$ and $L$ in $T$,
    \begin{equation}
      \label{eq:3}
      P(K,L)=\set{x\in G:K\cap x\cdot L\not=\emptyset}
    \end{equation}
    is compact in $G$.
  \item If $t_{i}\to t$ and $x_{i}\cdot t_{i}\to t'$ in $T$, then
    $\set{x_{i}}$ has a convergent subnet in $G$.
  \end{enumerate}
  If in addition $G$ acts freely on $T$, then the above conditions are
  equivalent to
  \begin{enumerate}[\rm (P{A}1)] \setcounter{enumi}{3}
  \item\label{PA4}
  If $t_{i}\to t$ and $x_{i}\cdot t_{i}\to t'$ in $T$, then
    $\set{x_{i}}$ is convergent in $G$.
  \end{enumerate}
\end{prop}
\begin{proof}
  This is a combination of \cite{wil:toolkit}*{Proposition~2.17 and
    Corollary~2.26.}.
\end{proof}

\begin{remark}
  It should be noted that this definition of a proper groupoid action
  is only appropriate for actions of locally compact
  groupoids on locally compact spaces.  In \cite{palais}, for
  \emph{group} actions, Palais writes $((K,L))$ in place of $P(K,L)$
  in Proposition~\ref{prop-proper-equiv}.  If $G$ and $T$ are
  completely regular, then Palais calls a $G$-space $T$ proper if each
  point has a neighborhood $S$ such that every point in $T$ has a
  neighborhood $U$ such that $((S,U))$ has compact closure
  \cite{palais}*{Definition~1.2.2}.  Our definitions coincide for
  locally compact group actions on locally compact (Hausdorff) spaces
  by \cite{palais}*{Theorem~1.2.9}.
We will discuss proper actions on \lchs s.
\end{remark}

Note that if $t=t'$ in (PA4), then we must have $x_{i}\to x$ for some
$x\in G$, and $x_{i}\cdot t_{i}\to x\cdot t$.  Since $T$ is Hausdorff,
$x\cdot t=t$ and $x\in G(t)=\set{x\in G:x\cdot t=t}$.  Thus we have
the following useful observation.

\begin{cor}
  \label{cor-principal-sc} Suppose that $T$ is a locally compact free
  and proper $G$-space with moment map $\rho:T\to\go$.  If
  $t_{i}\to t$ and $x_{i}\cdot t_{i}\to t$, then $x_{i}\to \rho(t)$.
\end{cor}

Recall from the
beginning of Section~3 of \cite{hkqwsemi}  our conventions
regarding \emph{$G$-bundles} $(T,p,B)$.

\begin{defn}\label{principal G space}
  We will call a $G$-space $T$ \emph{principal} if the
  $G$-action is both free and proper.  
Then a $G$-bundle $(T,p,B)$ is called
  \emph{principal} if $T$ is principal.\footnote{It should
    be emphasized that we \emph{do not} require our principal
    $G$-spaces to be locally
    trivial.}
\end{defn}

\begin{lem}\label{iso}
  If $G$ is a locally compact groupoid and $p:T\to B$ and $p':T'\to B$
  are principal $G$-bundles, then every $G$-bundle morphism
  $f:T\to T'$ over $B$ is an isomorphism.
\end{lem}
\begin{proof}
  To see that $f$ must be injective, suppose that $f(t)=f(t')$.  Then
  \begin{equation}
    \label{eq:4}
    p(t)=p'(f(t))=p'(f(t'))=p(t').
  \end{equation}
  Thus there is a unique $x\in G$ such that $t'=x\cdot t$.  But then
  $f(t)=f(t') =x\cdot f(t)$ and $x=\rho_{T}(t)$ since that action is
  free.  Thus $t=t'$ as required.

  To see that $f$ is surjective, let $t'\in T'$.  Since $p$ is
  surjective, there is a $t\in T$ such that $p(t)=p'(t')$.  Then
  $p'(f(t))=p(t)=p'(t')$. Hence there is an $x\in G$ such that
  $x\cdot f(t)=t'$.  But then $f(x\cdot t)=t'$.  Thus $f$ is
  surjective.

  To show that $f$ is open, we employ Fell's Criterion.  Suppose that
  $t_{i}'\to f(t)$ in $T'$.  Then $p'(t_{i}')\to p'(f(t))=p(t)$ in
  $B$.  Since $p$ is open, we can lift $\set{p'(t_{i}')}$ using
  Fell's Criterion (see \cite{hkqwsemi}*{Lemma~2.1}).
  Thus, after passing to a subnet and relabeling,
  we can find $t_{i}\to t$ in $T$ such that $p(t_{i})=p'(t_{i}')$.
  Then $p'(f(t_{i}))=p'(t_{i}')$ and there are $x_{i}\in G$ such that
  $x_{i}\cdot f(t_{i})=t_{i}'$.  But then $f(t_{i})\to f(t)$ while
  $x_{i}\cdot f(t_{i})\to f(t)$.  By
  Corollary~\ref{cor-principal-sc}, this forces
  $x_{i}\to \rho_{T'}(f(t))$.  But since $f$ is $G$-equivariant,
  $\rho_{T'}(f(t))=\rho_{T}(t)$.  Therefore $x_{i}\cdot t_{i}\to t$
  and $f(x_{i}\cdot t_{i})=t_{i}'$.
\end{proof}

\begin{lem}
  [Diagonal Actions] \label{diagonal} If $T$ and $U$ are $G$-spaces,
  then the diagonal action of $G$ on $T*_{r}U$ is principal if either
  $T$ or $U$ is principal.  
\end{lem}
\begin{proof}
This follows easily from item~\ref{PA4} of
  Proposition~\ref{prop-proper-equiv}: suppose $T$ is proper and that
  both $(t_{i},u_{i})\to (t,u)$ and
  $x_{i}\cdot (t_{i},u_{i})\to (t',u')$ in $T*U$.  Then $t_{i}\to t$
  and $x_{i}\cdot t_{i}\to t'$.  Then $\set{x_{i}}$ is convergent by
  \ref{PA4}.
\end{proof}

\begin{lem}
 \label{pullback}
 Let $(T,p,B)$ be a bundle and
  $f:C\to B$ a continuous map.  Then the pullback bundle
  $f^{*}T$ is principal
  whenever $T$ is.
\end{lem}
\begin{proof}
  If $T$ is principal, then we can use item~\ref{PA4} of
  Proposition~\ref{prop-proper-equiv} to see that $f^{*}T$ is
  principal just as in Lemma~\ref{diagonal}.
\end{proof}

We have the following universal characterization for
principal bundle pullbacks, which is modeled on that for groups due to
Palais mentioned in the introduction.  Our conventions for pull-backs
are recorded immediately preceding Lemma~3.2 in \cite{hkqwsemi}.

\begin{thm}[Palais]\label{palais}
  Suppose that $(f,g):(T,p,B)\to (U,q,C)$ is a morphism of principal
  $G$-bundles.  Then $T$ is isomorphic to the pullback $g^{*}U$.  In
  particular, the canonical map $h$ making the diagram
  \begin{equation}
    \label{eq:6}
    \begin{tikzcd}
      T \arrow[dd,"p",swap] \arrow[dr,"h",dashed,swap]
      \arrow[dr,"!",dashed] \arrow[rr,"f"] && U
      \arrow[dd,"q",swap]  \\
      & g^{*}U=B*U \arrow[ru,"\pi_{2}"] \arrow[dl,"\pi_{1}"] \\
      B\arrow[rr,"g"] &&C
    \end{tikzcd}
  \end{equation}
  commute is a $G$-bundle isomorphism over $B$.
\end{thm}
\begin{proof}
  The pullback $g^{*}U$ is principal by Lemma~\ref{pullback}.
  Moreover, $h$ is $G$-equivariant:
  \begin{align}
    \label{eq:7}
    h(x\cdot t)
    &=\bigl(p(x\cdot t),f(x\cdot t)\bigr)
      =\bigl(p(t),x\cdot f(t)\bigr)
    \\
    &=x\cdot \bigl(p(t),f(t)\bigr)
      =x\cdot h(t).
  \end{align}
  Thus $h$ is a $G$-bundle isomorphism by \lemref{iso}.
\end{proof}

\begin{defn}\label{bundle principal}
We say that an action  of $G$ on a bundle $p:T\to B$ over a locally compact
space $B$ is \emph{principal}
if the action on the base space $B$ is principal.
\end{defn}

\begin{remark} \label{rem-act-fell} Suppose that $G$ acts principally
  on a bundle $p:T\to B$ as above.  Then the $G$-action on $T$ must be
  free since $x\cdot t=t$ implies that $x\cdot p(t)=p(t)$.
  Furthermore, the action on $T$ must also satisfy item~\ref{PA4} of
  Proposition~\ref{prop-proper-equiv}: if $t_{i}\to t$ and
  $x_{i}\cdot t_{i}\to t'$, then $p(t_{i})\to p(t)$ and
  $x_{i}\cdot p(t_{i})\to p(t')$.  Hence $\set{x_{i}}$ converges since
  the $G$-action on $B$ is proper.  Hence if $T$ is locally compact,
  it is a principal $G$-space.  However, we want to allow for the case
  that $T$ might be neither locally compact nor
  Hausdorff---specifically, $T$ may be a Fell bundle.  But we still
  will have $B$ locally compact and Hausdorff.  Hence if $G$-acts
  principally on the bundle and $t_{i}\to t$ while
  $x_{i}\cdot t_{i}\to t$, then we have $x_{i}\to x$ and then
  $x_{i}\cdot t_{i}\to x\cdot t$.  But then $p(t)=x\cdot p(t)$, and
  since $B$ is Hausdorff, it follows that $x=\rho_{B}(p(t))=\rho_{T}(t)$, and we
  recover the conclusion of Corollary~\ref{cor-principal-sc} in this
  case.
\end{remark}

\section{Actions by isomorphisms}
\label{sec:actions}

\begin{defn}\label{act on groupoid}
Let $H$ be a locally compact groupoid, and let $G$ act on the space $H$.
We say the $G$-action on $H$ is
\emph{principal and by isomorphisms} if it is principal as an action
on the space $H$ and is also an action by isomorphisms. 
\end{defn}

\begin{rem}
Actions by isomorphisms are also discussed in \cite{busmeyfibration}.
\end{rem}

We want to extend \cite{kmqw2}*{Proposition~A.10} to principal
groupoid actions.  Recall that if $A$ and $B$ are subsets of a
groupoid $H$, then $AB=\set{ab:(a,b)\in A\times B\cap H\pairs}$.

\begin{lem}\label{lem-tech}
  Suppose that $G$ acts freely on $H$ by isomorphisms.   If
  $G\cdot s_{H}(h)=G\cdot r_{H}(k)$, then
  \begin{equation}
    \label{eq:11}
    (G\cdot h)(G\cdot k)=G\cdot (h'k')
  \end{equation}
for any $h'\in G\cdot s_{H}(h)$ and $k'\in G\cdot r_{H}(k)$ such that
$s_{H}(h')=r_{H}(k')$. 
\end{lem}
\begin{proof}
  Since we can replace $h$ by $h'$ and $k$ by $k'$ without changing
  the left-hand side of \eqref{eq:11}, we may as well assume
  $s(h)=r(k)$ from the start.   
We
  can assume (see \cite{hkqwsemi}*{Remark~5.2}) that $h,k,hk\in H_{u}$ where $u=\rho(s(h))=\rho(r(k))$, and where in turn
  $\rho:H\to \go$ is the moment map.  Thus $G\cdot h=G_{u}\cdot h$,
  $G\cdot k=G_{u}\cdot k$ and $G\cdot (hk)=G_{u}\cdot (hk)$.  Since
  \begin{equation}
    \label{eq:12}
    x\cdot (hk)=(x\cdot h)(x\cdot k),
  \end{equation}
  we have $G\cdot (hk) \subset (G\cdot h)(G\cdot k)$.

  On the other hand, if $s(x\cdot h)=r(y\cdot k)$, then
  \begin{equation}
    \label{eq:13}
    x\cdot s(h)=y\cdot r(k).
  \end{equation}
Since $s(h)=r(k)$ and $G$ acts freely, we must have $x=y$, therefore
$(x\cdot h)(x\cdot k)=x\cdot (hk)\in G\cdot (hk)$.   Thus $(G\cdot
h)(G\cdot k)\subset G\cdot (hk)$ and \eqref{eq:11} holds.
\end{proof}

Let
\begin{equation}
  \label{eq:14}
  (G\backslash H)\pairs=\set{(G\cdot h,G\cdot k)\in G\backslash H
    \times G\backslash H: G\cdot s(h)=G\cdot r(k)}.
\end{equation}
Using Lemma~\ref{lem-tech}, we get a well-defined map from $(G\backslash
H)\pairs$ to $G\backslash H$ sending $(G\cdot h, G\cdot k)\to (G\cdot
h)(G\cdot k)$.  We can also define $(G\cdot h)^{-1}=G\cdot h^{-1}$.

\begin{prop}
  \label{prop-groupoid}  Suppose that $G$ acts principally on $H$ by
  isomorphisms.   Let $(\guh)\pairs$ be as in \eqref{eq:14}.  Then
  with multiplication and inversion defined as above, and assuming $G$
  has an open range map, $\guh$ is a
  locally compact groupoid with unit space
  $(\guh)\units=G\backslash\ho$.   Then range and source maps on
  $\guh$ are given by $r(G\cdot h)=G\cdot r_{H}(h)$ and $s(G\cdot
  h)=G\cdot s_{H}(h)$.   If $H$ has an open range map, then so does
  $\guh$. 
\end{prop}
\begin{proof}
  It is routine to verify that $\guh$ satisfies axioms (a), (b), and
  (c) of \cite{wil:toolkit}*{Definition~1.2}.  Hence $\guh$ is a
  groupoid as claimed.  Furthermore $r_{\guh}(G\cdot h)=(G\cdot
  h)(G\cdot h)^{-1}=G\cdot hh^{-1}=G\cdot r(h)$.   Similarly
  $s_{\guh}(G\cdot h)=G\cdot s(h)$, and we have
  $(\guh)\units=G\backslash\ho$. 

  Since the action of $G$ on $H$ is proper, $\guh$ is locally compact
  Hausdorff by \cite{wil:toolkit}*{Proposition~2.18}.

  To see that $\guh$ is a topological groupoid, we need to verify that
  the operations are continuous.  To that end, suppose that $(G\cdot
  h_{i},G\cdot k_{i})\to (G\cdot h,G\cdot k)$ in $(\guh)\pairs$.
  There is no harm in assuming $s(h_{i})=r(k_{i})$ and that
  $s(h)=r(k)$.   Then we want to see that $G\cdot (h_{i}k_{i})\to
  G\cdot (hk)$ in $\guh$.  Since it will suffice to show that every
  subnet has a subnet converging to $G\cdot (hk)$, we can replace
  $\set{G\cdot (h_{i}k_{i})}$ with a subnet, relabel, and show that
  $\set{G\cdot (h_{i}k_{i})}$ has a convergent subnet.   Since $G$ has
  open range and source maps, the orbit map from $G$ to $\guh$ is
  open.  Hence we can pass to a subnet, relabel, and assume there are
  $x_{i},y_{i}\in G$ such that $x_{i}\cdot h_{i}\to h$ and $y_{i}\cdot
  k_{i}\to k$.  Letting $s_{H}(h_{i})=v_{i}=r_{H}(k_{i})$, we must have
  $x_{i},y_{i}\in G_{u_{i}}$ where $u_{i}=\rho(v_{i})$ as in the proof of
  Lemma~\ref{lem-tech}.

  Then we have $x_{i}\cdot v_{i}\to v$ while $y_{i}\cdot v_{i}\to v$
  where $s_{H}(h)=v=r_{H}(k)$.   
  But then $y_{i}\cdot v_{i}=(y_{i}x_{i}^{-1})x_{i}\cdot v_{i}\to v$.
  It follows from Corollary~\ref{cor-principal-sc} that
  $y_{i}x_{i}^{-1}\to \rho(v)$.   Thus $y_{i}\cdot
  h_{i}=(y_{i}x_{i}^{-1})x_{i}\cdot h_{i}\to h$.   But then
  $y_{i}\cdot (h_{i}k_{i})=(y_{i}\cdot h_{i})(y_{i}\cdot k_{i}) \to
  hk$.   Thus $G\cdot (h_{i}k_{i})\to G\cdot (hk)$ as required.   The
  continuity of inversion is similar, but more straightforward.

  To see that the range map on $\guh$ is open, we use Fell's
  Criterion.  Suppose that $G\cdot v_{i}\to G\cdot r_{H}(h)=r(G\cdot h)$.  Since
  the range map on $G$ is open, we can pass to a subnet, relabel, and
  assume that there are $x_{i}\in G$ such that $x_{i}\cdot v_{i}\to
  r_{H}(h)$.   Since $r_{H}$ is open, we can pass to another subnet
  and find $h_{i}\to h$ such that $r_{H}(h_{i})=x_{i}\cdot v_{i}$.
  Then $G\cdot h_{i}\to G\cdot h$ and $r(G\cdot h_{i})=G\cdot v_{i}$. 
\end{proof}

\begin{rem}
Observe that when $G$ acts principally and by isomorphisms on $H$,
the $G$-bundle map $H\to G\under H$
is a surjective groupoid homomorphism.
\end{rem}

\begin{example}
  \label{ex-g-act-quotient} Let $H$ be the action groupoid
  for the right action of $G$ on itself as in 
  \cite{hkqwsemi}*{Example~5.6}.
  Then $(x,y,xy)\mapsto y$ induces an isomorphism of $\guh$ onto~$G$.
\end{example}

\begin{prop}
  \label{prop-haar-guh} Suppose that $H$ has a Haar system $\lambdah$,
  that $G$ acts principally on $H$ by isomorphisms, and that the
  action is invariant as in 
  \cite{hkqwsemi}*{Definition~5.7}.  Then
  $\guh$ has a Haar system $\lambdaguh$ given by
  \begin{equation}
    \label{eq:16}
    \int_{\guh} g(G\cdot h)\,d\lambdaguh^{G\cdot v}(G\cdot h) =
    \int_{H} g(G\cdot h)\,d\lambdah^{v}(h)
  \end{equation}
for $g\in C_{c}(\guh)$ and $v\in\ho$.
\end{prop}

\begin{proof}
  Let $r_{H}:H\to \ho$.  Then $\lambdah=\set{\lambdah^{v}}_{v\in\ho}$
  is an equivariant $r_{H}$-system.   Thus
  $\lambdaguh=\set{\lambdaguh^{G\cdot v}}_{G\cdot v\in \guh} $ is
  $r_{\guh}$-system by \cite{wil:toolkit}*{Proposition~3.14}.  The
  rest is routine.
\end{proof}

\begin{defn}
If an action of $G$ on a Fell bundle $p:\AA\to H$ is
principal as a bundle action and is also an action by isomorphisms, we
say that $G$ acts \emph{principally and by isomorphisms}.
\end{defn}


Assume that $p:\AA\to H$ is a Fell bundle and that $G$ acts on $\AA$
principally and by isomorphisms as above.  Then $\guh$ is a locally
compact Hausdorff groupoid by Proposition~\ref{prop-groupoid}.  Since
$p$ is equivariant, we have a continuous map
\begin{align}
  \label{eq:19}
  \pb:\gua\to\guh
\end{align}
given by $G\cdot a\mapsto G\cdot p(a)$.  Here $\gua$ is the orbit
space with the quotient topology.\footnote{Note that $\AA$ is almost
  never locally compact---the relative topology on the fibres is the
  Banach space topology---and may not even be Hausdorff (see
  \cite{danacrossed}*{Example~C.27}).}  Using Fell's Criterion, we can
see that $\pb$ is also open: suppose
$G\cdot h_{i}\to G\cdot p(a)=\pb(a)$ in $\guh$.  Then we may as well
assume that $h_{i}\to p(a)$.  Since $p$ is open, we can pass to
subnet, relabel, and assume that there are $a_{i}\to a$ in $\AA$ such
that $p(a_{i})=h_{i}$.  Then $G\cdot a_{i}\to G\cdot a$ and
$\pb(G\cdot a_{i})=G\cdot h_{i}$.

To see that $\pb:\gua\to\guh$ is an upper semicontinuous Banach
bundle, we first have to equip each fibre $\pb^{-1}(G\cdot h)$ with a
Banach space structure.  But $a\mapsto G\cdot a$ is a bijection of
$p^{-1}(h')=A(h')$ onto $\pb^{-1}(G\cdot h)$ for any $h'\in G\cdot h$.
Since $G$ acts by isomorphisms as in
\cite{hkqwsemi}*{Definition~5.9(b)},
we can use this map to impose a
well-defined Banach space structure on $\pb^{-1}(G\cdot h)$:
define $\| G\cdot a\|=\|a\|$, $G\cdot a+ G\cdot b= G\cdot (a+b)$, and
$\alpha G\cdot a=G\cdot (\alpha a)$.  Since $G$ acts by isomorphisms,
these operations are independent of our choice of $h'\in G\cdot h$.

\begin{prop}
  \label{lem-gua-usc} Suppose that $p:\AA\to H$ is a Fell bundle on
  which $G$ acts principally and by isomorphisms.  Let $\pb:\gua\to\guh$
  be as above.  Then the 
map
  $\pb:\gua\to\guh$ 
  introduced above
  is a Fell bundle with operations
  $(G\cdot a)(G\cdot b)=G\cdot (ab)$ provided $s(a)=r(b)$, and
  $(G\cdot a)^{*}=G\cdot a^{*}$.
\end{prop}
\begin{proof}
  First, to see that $\bar p:\gua\to\guh$ is an upper-semicontinuous Banach bundle,
  we need to verify axioms (B1)--(B4) of
  \cite{hkqwsemi}*{Definition~2.4}.  This can be done almost exactly as
  in \cite{kmrw}*{Proposition~2.15} with minor alterations.  To verify
  (B1), suppose that $G\cdot a_{i}\to G\cdot a$ with
  $\|a_{i}\|\ge \epsilon>0$ for all $i$.  Then, after passing to a
  subnet and relabeling, we can assume that there are $x_{i}\in G$
  such that $x_{i}\cdot a_{i}\to a$.  Since the action of $G$ must be
  isometric, we have $\|a\|\ge\epsilon$.  This verifies (B1).

  (B2): Suppose that
  $(G\cdot a_{i},G\cdot b_{i})\to (G\cdot a,G\cdot b)$ in
  $\gua*\gua$. Then
  $\pb(G\cdot a_{i})= G\cdot p(a_{i})=G\cdot p(b_{i})= \pb(G\cdot
  b_{i})$ for all $i$.  By adjusting $a_{i}$ and $b_{i}$ within their
  orbits, we can assume that $p(a_{i})=p(b_{i})$ and $p(a)=p(b)$.
  Thus we need to see that $G\cdot (a_{i}+b_{i})\to G(a+b)$ in $\gua$.
  If this is not the case, then we can pass to a subnet, relabel, and
  assume that there is a neighborhood $U$ of $G\cdot (a+b)$ that
  contains no $G\cdot (a_{i}+b_{i})$.  Since $G$ has open
  range and source maps, the orbit map from $\AA$ to $\gua$ is open
  \cite{wil:toolkit}*{Proposition~2.12}.  Hence we can pass to another
  subnet, relabel, and assume that there are $x_{i}\in G$ such that
  $x_{i}\cdot (a_{i}+b_{i})\to (a+b)$.  But then
  $G\cdot (a_{i}+b_{i})\to G\cdot (a+b)$ which leads to a
  contradiction.  This proves (B2).

  The validity of (B3) and (B4) follow by similar, albeit easier,
  arguments.
  
    To see that the indicated operations give a Fell bundle,
    we first need to check that the multiplication is well-defined:
    Suppose that $(G\cdot a,G\cdot b)\in\gua\pairs$; that is,
$(G\cdot p(a),G\cdot p(b))\in\guh\pairs$.  As in the proof of
Lemma~\ref{lem-tech}, we can adjust $p(a)$ and $p(b)$ in their orbits,
and hence $a$ and $b$ in theirs, so that $(p(a),p(b))\in H\pairs$.
Then the ``product'' $(G\cdot a)(G\cdot b)$ is the orbit
$G\cdot (ab)$.  Hence, just as in Proposition~\ref{prop-groupoid}, we
get a well-defined map from $\gua\pairs\to\gua$ sending
$(G\cdot a,G\cdot b)$ to $G\cdot (ab)$.
    Properties~(FB1)--(FB3)
  of 
  \cite{hkqwsemi}*{Definition~2.5}
  are
  routine.  Since $a\mapsto G\cdot a$ is an isomorphism of $A(h)$ onto
  $(pb)^{-1}(G\cdot h)$, 
  (FB4)--(FB5)
  follow since $A(h)$
  is given to be an $A(r(h))\sme A(s(h))$-\ib.
\end{proof}

\begin{lemma}
  \label{G on pullback}
  Let a \lchs\ $T$ have commuting actions by \lcg s $G$ and $K$ such
  that the moment map $\rho_K:T\to K\units$ is a principal $G$-bundle.
  Then we can view $\ac KT=K*T$ as the pullback by
  $s_{K}:K\to K\units$.
  \begin{enumerate}[\rm (a),nosep]
  \item The pullback action $G\actleft \ac KT$ given by
    $x\cdot (k,t)=(k,x\cdot t)$ is principal and by isomorphisms.
  \item The map $\theta:G\backslash(\ac KT)\to K$ given by
    $\theta([k,t])=k$ is a groupoid isomorphism.
  \item The right action of $K$ on $\ac KT$ given by $(k,t)\cdot \ell=(k\ell,\ell^{-1}\cdot t)$, with moment map
    $\sigma(k,t)=s(k)$, is principal.
  \end{enumerate}
\end{lemma}
\begin{proof}
  For part (a), it follows from \lemref{pullback} that the pullback
  $G$-action is principal, and it is easy to see that the action is by
  isomorphisms:
  \begin{align}
    \label{eq:24}
    x\cdot (k\ell,t)
    &= (k\ell,x\cdot t) = (k,\ell\cdot (x\cdot t))(\ell,x\cdot t) \\
    &= (k,x\cdot (\ell\cdot t))(\ell,x\cdot t) = \bigl(x\cdot (k,\ell\cdot
      t)\bigr) \bigl(x\cdot (\ell,t)\bigr).
  \end{align}

  For part (b), note that the coordinate projection
  $\pi_1:(k,t)\mapsto k$ factors through $G\backslash (\ac KT)$ and we
  get a continuous bijection $\theta$ making the diagram
  \begin{equation}
    \label{eq:25}
    \begin{tikzcd}
      & \ac KT \arrow[dl,"\pi_{1}",swap] \arrow[dr,"q"] \\
      K && \arrow[ll,"\theta",swap] G\backslash(\ac KT)
    \end{tikzcd}
  \end{equation}
  commute.  
  Then $\theta$ is a homeomorphism because $\pi_{1}$ and $q$ are open, and is a groupoid isomorphism because $\pi_{1}$ and $q$ are groupoid homomorphisms.
  
  Part~(c) is an immediate consequence of Lemma~\ref{diagonal}
  (adapted to right actions).
\end{proof}

Our next result shows that the situation in Lemma~\ref{G on pullback}
is generic for principal groupoid actions.  This result
parallels the corresponding result in the group case proved in
\cite{kmqw2}*{Theorem~A.15}.  Note that the unit space of the action
groupoid $\ac HT$ can be identified with $T$.

\begin{thm}\label{classify}
  Let $G$ and $H$
  be locally compact groupoids, let $G$ act principally and by
  isomorphisms on $H$, and let $K=G\under H$ be the quotient groupoid.
  Then there is an action of $K$ on $H\units$ such that
  \begin{enumerate}[\rm (a),nosep]
  \item $H$ is $G$-equivariantly isomorphic to the action groupoid
    $\ac K{H\units}$, and
  \item the restricted action $G\actleft H\units$ commutes with the
    action of $K$.
  \end{enumerate}
\end{thm}

\begin{proof}
  Let $q:H\to K$ be the quotient map, and $q|:H\units\to K\units$ the
  restriction to unit spaces.  We have a commuting diagram
  \begin{equation}
    \label{eq:26}
    \begin{tikzcd}
      \ho\arrow[d,"q|",swap] & H \arrow[l,"r_{H}",swap] \arrow[r,"s_{H}"]
      \arrow[d,"q"] &
      \ho \arrow[d,"q|"] \\
      K\units & K \arrow[l,"r_{K}"] \arrow[r,"s_{K}",swap] & K\units
    \end{tikzcd}
  \end{equation}
  where $(s_{H},s_{K})$ and $(r_{H},r_{K})$ are principal $G$-bundle
  morphisms.  So by \thmref{palais}, we can form the pullbacks
  $r_{K}^{*}K\units=K*_{r}\ho=\set{(k,u):r_{K}(u)=q|(u)}$ and
  $s_{K}^{*}K=K*_{s}\ho =\set{(k,u:s_{K}(k)=q|(u)}$, and the canonical
  $G$-bundle isomorphisms $\theta_r$ and $\theta_s$ such that the
  diagram
  \begin{equation}
    \label{eq:27}
    \begin{tikzcd}
      \ho\arrow[dd,"q|",swap] && H \arrow[dd,"q"]
      \arrow[ll,"r_{H}",swap] \arrow[rr,"s_{H}"]
      \arrow[dl,"\theta_{r}",swap] \arrow[dl,"\cong"]
      \arrow[dr,"\theta_{s}"] \arrow[dr,"\cong",swap] &&\ho
      \arrow[dd,"q|"] \\
      &K*_{r}\ho \arrow[ul,"\pi_{2}",swap ] \arrow[dr,"\pi_{1}",swap ]
      && K*_{s}
      \ho \arrow[ur,"\pi_{2}"] \arrow[dl,"\pi_{1}"] \\
      K\units&& K \arrow[ll,"r_{K}"] \arrow[rr,"s_{K}",swap] && K\units
    \end{tikzcd}
  \end{equation}
  commutes.  At this point, the only structure on the pullbacks
  $K*_{s}\ho$ and $K*_{r}\ho$ is the $G$-bundle structure coming from
  Theorem~\ref{palais}.  We use the homeomorphisms $\theta_r$ and
  $\theta_s$ to impose groupoid structures on these pullbacks.  Then,
  by definition,
  \[
    \theta:=\theta_r\circ\theta_s\inv:K*_sH\units\to K*_rH\units
  \]
  is an isomorphism.  If $\theta(k,u)=(k',u')$, then
  $k'=\pi_{1}\circ \theta_{r}\circ
  \theta_{s}^{-1}=q(\theta_{s}^{-1}(k,u))=k$.  It follows that
  $\theta$ has the form
  \[
    \theta(k,u)=(k,k\cdot u)
  \]
  for some continuous map
  \[
    (k,u)\mapsto k\cdot u:K*_sH\units\to H\units.
  \]

  We claim that this gives an action of the groupoid $K$ on the space
  $H\units$ with respect to the moment map
  $\rho=q|:\ho\to K\units$.\footnote{Note that
    $q|:\ho\to K\units=G\backslash \ho$ is open since $G$ has an open
    range map.}  Thus we need to show that $q|(u)\cdot u=u$, and that
  $k\cdot (\ell\cdot u)=(k\ell)\cdot u$ for composable $k$, $\ell$,
  and $u$.

  We claim that in the imposed groupoid structure on $K*_{s}\ho$,
  \begin{align}
    \label{eq:30}
    s(k,u)=\bigl(s_{K}(k),u\bigr)\quad \text{and} \quad
    r(k,x)=\bigl(r_{K}(k),k\cdot u\bigr). 
  \end{align}

  To verify the formula for the range map, we use the commutativity of
  \eqref{eq:27}: let $h=\theta_{s}^{-1}(k,u)$.  Then
  \begin{align*}
    \pi_2(r(k,u))
    &=\pi_2(r(\theta_s(h)))
    \\&=\pi_2(\theta_s(r_{H}(h)))
    \\&=s_{H}(r_{H}(h))
    \\&=r_{H}(h)
    \\&=\pi_2(\theta_r(h))
    \\&=\pi_2(\theta_r\circ\theta_s\inv(k,u))
    \\&=k\cdot u.
  \end{align*}
  Similarly,
  \begin{align}
    \label{eq:29}
    \pi_1(r(k,u))&=\pi_{1}(r(\theta_{s}(h)))
    \\
                 &=\pi_{1}(\theta_{s}(r_{H}(h)))\\
                 &
                   = q(r_{H}(h))=q|(r_{H}(h))\\
                 &=r_{K}(q(h))=r_{K}(k).
  \end{align}

  The formula for the source map is proved similarly.

  Then $\bigl((k,u),(\ell,v)\bigr)\in (K*_{s}\ho)\pairs$ if and only
  if $(k,\ell)\in K\pairs$ and $u=\ell\cdot v$.  We claim
  \begin{align}
    \label{eq:28}
    (k,u)(\ell,v)=(k,\ell\cdot v)(\ell,v)=(k\ell,v).
  \end{align}
  The verify the claim, let $(k,u)=\theta_s(h)$ and
  $(\ell,v)=\theta_s(k)$.  We have
  \begin{align*}
    \pi_1\bigl((k,u)(\ell,v)\bigr)
    &=\pi_1\circ\theta_s(hk)
      =q(hk)=q(h)q(k)=k\ell,
  \end{align*}
  and similarly
  \[
    \pi_2\bigl((k,u)(\ell,v)\bigr) =\pi_2\circ\theta_s(hk)=
    s_{H}(hk)=s_{H}(k)=v.
  \]
  This establishes \eqref{eq:28}.

  We now see that
  \[
    r(k,\ell\cdot u) =r\bigl((k,\ell\cdot u)(\ell,u)\bigr)
    =r(k\ell,u),
  \]
  consequently
  \[
    k\cdot (\ell\cdot u)=(k\ell)\cdot u.
  \]
  Since $s(q|(u),u)=(q|(u),u)$, we must also have
  $(q|(u),u)=r(q|(u),u)=(q|(u),q|(u)\cdot u)$, so $q|(u)\cdot u=u$ and
  we have an action of $K$ on $\ho$.  
  Moreover, $\theta_{s}:H\to K*_{s}\ho$ is an isomorphism of locally
  compact groupoids.
  
  We verify that this isomorphism is $G$-equivariant:
  \begin{align*}
    x\cdot \theta_s(h)
    &=x\cdot \bigl(q(h),s_{H}(h)\bigr)
      =\bigl(q(h),x\cdot s_{H}(h)\bigr)
      =\bigl(q(h),s_{H}(x\cdot h)\bigr)
    \\&=\bigl(q(x\cdot h),s_{H}(x\cdot h)\bigr)
    =\theta_s(x\cdot h).
  \end{align*}

  This establishes part~(a), and then~(b) follows from the computation
  \begin{align*}
    \bigl(r_{H}(h),h\cdot (x\cdot t)\bigr)
    &=r(h,x\cdot t)
    \\&=r\bigl(x\cdot (h,t)\bigr)
    \\&=x\cdot r(h,t)
    \\&=x\cdot \bigl(r_{H}(h),h\cdot t\bigr)
    \\&=\bigl(r_{H}(h),x\cdot (h\cdot t)\bigr).
    \qedhere
  \end{align*}
\end{proof}

When $G$ acts principally and by isomorphisms on a \lcg\ $H$, we will
find it useful to have formulas describing various constructions when
$H$ is replaced by an action groupoid $\ac KT$ as described in
\thmref{classify}.  Moreover, it is convenient to apply \lemref{G on
  pullback} along with \thmref{classify}.  For example, we already
mentioned in \lemref{G on pullback} that
\[
  G\under (\ac KT)\simeq K
\]
via the map $G\cdot (k,t)\mapsto k$.

\begin{rem} If $h\in H$ and $q:H\to K$ is the quotient map, then
  $\theta_{s}(h)=\bigl(q(h),s_{H}(h)\bigr)$.  Thus
  $\theta(q(h),s_{H}(h))= \theta_{r}(h)=\bigl(q(h),r_{H}(h)\bigr)$.
  Hence the action of the quotient groupoid $K=G\under H$ on $H\units$
  in the previous result satisfies
  \[
    r_{H}(h)=q(h)\cdot s_{H}(h)\righttext{for}h\in H.
  \]
\end{rem}

We want to promote Theorem~\ref{classify} to a classification of
principal actions by isomorphisms on Fell bundles, thus providing a
generalization to groupoid actions of
\cite[Theorem~A.16]{kmqw2}.
For the proof, we will need to introduce the concept of an
\emph{action Fell bundle}.  Let $\pb:\BB\to K$ be a Fell bundle over a
groupoid $K$ and assume that $K$ acts on a space $T$.  As usual, we
let $\ac KT=K*T$ be the action groupoid with unit space identified with
$T$ so that
$r(k,t)=k\cdot t$,  $s(y,t)=t$, and
$(k,\ell\cdot t)(\ell,t)=(k\ell,t)$.  Then $(k,t)\mapsto k$ is a
groupoid homomorphism and we can form the pullback Fell bundle
\begin{equation}
  \label{eq:31}
  \pi_{1}^{*}\BB=\set{(b,k,t):\pb(b)=k}
\end{equation}
over $\ac KT$.  Clearly, we can condense the notation and consider
elements of $\pi_{1}^{*}\BB$ to be pairs
$\ac \BB T=\set{(b,t):s(\pb(b))=\rho(t)}$.  Then $s(b,t)=t$ and
$r(b,t)=\pb(b)\cdot t$.  Moreover,
\begin{equation}
  \label{eq:32}
  (b,\pb(b')\cdot t)(b',t)=(bb',t)\quad\text{and}\quad
  (b,t)^{*}=(b^{*},\pb(b)\cdot t).
\end{equation}

Then in analogy with Lemma~\ref{G on pullback}, we have the following.

\begin{lem}\label{G on pullback Fell}
  Let $G$ be a groupoid, $p:\BB\to K$ a Fell bundle, and suppose that we have a locally compact Hausdorff space $T$ which admits commuting actions by $G$
  and $K$ such that the moment map
  $\rho_K:T\to K\units$ is a principal $G$-bundle
  \(as in \lemref{G on pullback}\).
  \begin{enumerate}[\rm (a),nosep]
  \item The pullback action $G\actleft \ac \BB T$ given by
    $x\cdot (b,t)=(b,x\cdot t)$ is principal and by isomorphisms.
  
  \item The map $\Theta:G\under(\ac \BB T)\to \BB$ given by
    $\Theta([b,t])=b$ is a Fell-bundle isomorphism.
  
  \item $\BB$ acts on the right of $\ac \BB T$ by
    $(b,t)\cdot c=(bc,p(c)\inv\cdot t)$.
  \end{enumerate}
\end{lem}

\begin{remark}
  Note that the statement in item~(c) above includes the assertion
  that the action 
$(\ac \BB T)\actright \BB$ covers the principal groupoid action
$(\ac KT)\actright K$ of \lemref{G on pullback}.
\end{remark}

\begin{proof}[Sketch of the Proof]
  The proof of (a)--(b) follows the same lines as \lemref{G on
    pullback}.
  For part~(c), we check the conditions in the discussion preceding
  \cite{mw:fell}*{Definition~6.1}: first of all, the pairing 
  \[
    (\ac \BB T)*\BB\to \ac \BB T
  \]
  is continuous since the left coordinate is the multiplication in
  $\BB$ and the right coordinate is the action of $K$ on $T$.  Routine
  computations show that the pairing covers the action of $K$ on $\ac
  KT$, $((b,t)\cdot c)\cdot d=(b,t)\cdot (cd)$, and $\|(b,t)\cdot
  c\|\le \|(b,t)\|\|c\|$. 
\end{proof}

Even though we did not think it necessary to give any details in the
above proof, it will be convenient to have the lemma for
reference. For example, the isomorphism $G\under(\ac \BB T)\simeq \BB$
is useful in our proof of the following classification theorem.

\begin{thm}\label{classify Fell}
  Let a locally compact groupoid $G$ act principally and by
  isomorphisms on a Fell bundle $p:\AA\to H$, and let $\pb:\BB\to K$
  be the quotient Fell bundle $G\under \AA\to G\under H$.  Then there is
  an action of $K$ on $H\units$ such that $\AA$ is $G$-equivariantly
  isomorphic to the action Fell bundle
  \[
    \ac \BB H\units\to \ac KH\units.
  \]
\end{thm}

\begin{proof}
  Let $\bar q:\AA\to\BB$ and $q:H\to K$ be the quotient maps.  As
  above we can form the action Fell bundle $\ac \BB{\ho}$ for the
  $K$-action on $\ho$ from Theorem~\ref{palais}.  Then
  \begin{align}
    \label{eq:35}
    \ac \BB{\ho}=\set{(b,v)\in\BB\times \ho:s_{K}(\pb(b))=q|(v)}.
  \end{align}
  Thus, topologically at least, $\ac \BB{\ho}$ is the pullback
  $(s_{K}\circ\pb)*K\units$:
  \begin{equation}
    \label{eq:36}
    \begin{tikzcd}[row sep=1.5cm, column sep =2cm] 
      \ac \BB{\ho}\arrow[r,"\pi_{2}"] \arrow[d,"\pi_{1}",swap] & \ho
      \arrow[d,"q|"] \\
      \BB\arrow[r,"s_{K}\circ \pb",swap]& K\units .
    \end{tikzcd}
  \end{equation}
  
  Since $(\ho,q|,K\units)$ is a principal $G$-bundle, it follows
  exactly as in the proof of Lemma~\ref{pullback} that
  $(\ac \BB{\ho},\pi_{1},\BB)$ is a $G$-bundle with respect to the
  action of $G$ on its second factor.  Then we get a commutative
  diagram
  \begin{equation}
    \label{eq:33}
    \begin{tikzcd}[row sep=1cm, column sep=2cm]
      \AA\arrow[rr,"s_{H}\circ p"] \arrow[dr,"\Theta"]
      \arrow[dr,"!",swap] \arrow[dd,"\bar q",swap] && \ho
      \arrow[dd,"q|"] \\
      &\ac \BB{\ho} \arrow[ur,"\pi_{2}"] \arrow[dl,"\pi_{1}"] \\
      \BB\arrow[rr,"s_{K}\circ \pb",swap] &&K\units
    \end{tikzcd}
  \end{equation}
  where $\Theta$ is the uniquely determined map given by
  $\Theta(a)=\bigl(q(a),s_{H}(p(a))\bigr)$.  We claim that $\Theta$ is
  a $G$-equivariant homeomorphism.  It is clearly $G$-equivariant:
  since $G$ acts by automorphisms of $H$, we have
  $\Theta(x\cdot a)=\bigl(q(x\cdot a),s_{H}(p(x\cdot a))\bigr)=
  \bigl(q(a),x\cdot s_{H}(p(a))\bigr)=x\cdot\Theta(a)$.

  The proof that $\Theta$ is a bijection follows exactly as in the
  proof of Lemma~\ref{iso}.  Since $\Theta$ is clearly continuous, we
  just need to verify it is open.  For this, since our bundles are not
  locally compact or necessarily Hausdorff, we have to modify the
  proof of Lemma~\ref{iso} slightly.  We still use Fell's Criterion.
  Suppose that
  $(b_{i},v_{i})\to \Theta(a)=\bigl(\bar q(a),s_{H}(p(a))\bigr)$ in
  $\ac \BB{\ho}$.  Since $\bar q$ is open, we can pass to a subnet,
  relabel, and assume there are $a_{i}\to a$ in $\AA$ such that
  $\bar q(a_{i})=b_{i}$.  Since $\pi_{1}(\Theta(a_{i}))=b_{i}$, there
  are $x_{i}\in G$ such that $(b_{i},v_{i})=x_{i}\cdot \Theta(a_{i})$.
  Then $x_{i}\cdot s_{H}(p(a_{i}))\to s_{H}(p(a_{i}))$ and
  $s_{H}(p(a_{i}))\to s_{H}(p(a))$.  Since the action of $G$ on $\ho$
  is principal, we have
  $x_{i}\to \rho_{H}(s_{H}(p(a)))=\rho_{H}(p(a))= \rho_{\AA}(a)$.
  Hence $x_{i}\cdot a_{i}\to a$ and
  $\Theta(x_{i}\cdot a_{i})=(b_{i},v_{i}) $.  Therefore $\Theta$ is
  a homeomorphism.

  Then we get a commutative diagram
  \begin{equation}
    \label{eq:34}
    \begin{tikzcd}
      \AA\arrow[r,"\Theta"] \arrow[d,"p",swap] & \ac \BB{\ho}
      \arrow[d,,"\pb\times \id"] \\
      H\arrow[r,"\theta_{s}",swap] &K*_{s}\ho
    \end{tikzcd}
  \end{equation}
  where $\pb\times \id(b,u)=\bigl(\pb(b),u)$ and
  $\theta_{s}(h)=\bigl(\bar q(h),s_{H}(h)\bigr)$ is the groupoid
  isomorphism from the proof of Theorem~\ref{classify}.  Since
  $\Theta$ is a $G$-equivariant homeomorphism and $\theta_{s}$ is a
  groupoid isomorphism, $\Theta$ will be a Fell bundle isomorphism
  provided it preserves multiplication and the involution.

  Since
  $\theta_{s}(p(b))=\bigl(\bar q(p(b)),s_{H}(p(b))\bigr)=
  \bigl(\pb(q(b)), s_{h}(p(b))\bigr)$, it follows that
  $r_{H}(p(b))=\pb(q(b))\cdot s_{h}(p(b))$.  Hence, if $a,b\in\AA$
  with $s\circ p(a)=r\circ p(b)$, then
  \begin{align*}
    \Theta(a)\Theta(b)
    &=\bigl(q(a),s\circ p(a)\bigr)\bigl(q(b),s\circ p(b)\bigr)
    \\&=\bigl(q(a)q(b),s\circ p(b)\bigr)
    \\&=\bigl(q(ab),s\circ p(ab)\bigr)
    \\&=\Theta(ab),
  \end{align*}
  where the multiplication on the right-hand side of the top line is
  defined since
  \[
    s_{H}(p(a))=r_{H}(p(b))=\pb(q(b))\cdot s_{H}(p(b)).
  \]
  For the involution, we have
  \begin{align*}
    \Theta(a)^*
    &=\bigl(q(a),s_{H}(p(a))\bigr)^*
    \\&=\bigl(q(a)^*,\pb(q(a))\cdot s_{H}(p(a))\bigr)
    \\&=\bigl(q(a^*),r_{H}(p(a))\bigr)
    \\&=\bigl(q(a^*),s_{H}(p(a)\inv)\bigr)
    \\&=\bigl(q(a^*),s_{H}(p(a^*))\bigr)
    \\&=\Theta(a^*).
    \qedhere
  \end{align*}
\end{proof}

\section{Semidirect-product actions}

Give an action of $G$ on $H$ by isomorphisms, we want to consider how
to build actions of their semidirect product $\sd HG$ on a space $T$
from actions of $H$ and $G$ on $T$.   For motivation, consider the
case where $H$ and $G$ are groups acting on a space $T$ and $\sd HG$
is the semidirect product of groups.   The following notation may be
helpful.  Let $\alpha:G\to \Aut H$ be a homomorphism and let $L(x)$ and
$\pi(h)$ be the homeomorphisms of $T$ induced by $x\in G$ and $h\in H$
respectively.   It is basic algebra to verify that we get an action of
$\sd HG$ restricting to the given actions on $H$ and $G$ exactly when
the pair $(\pi,L)$ is \emph{covariant} in that the diagram
\begin{equation}
  \label{eq:50}
  \begin{tikzcd}[column sep=2cm]
    T\arrow[r,"L(x)"] \arrow[d,"\pi(h)",swap] & T
    \arrow[d,"\pi(\alpha_{x}(h))"] \\
    T\arrow[r,"L(x)",swap] & T
  \end{tikzcd}
\end{equation}
commutes for all $x\in G$ and $h\in H$.  Alternatively,
\begin{align}
  \label{eq:51}
  L(x)\circ \pi(h) = \pi\bigl(\alpha_{x}(h)\bigr) \circ L(x),
\end{align}
which is a formula that is familiar to that
for covariant representations of dynamical
systems.   (This is the motivation for our terminology below.)

Now we return to the case where $G$ and $H$ are groupoids.  We let
$\rho_{G}:T\to \go$ and $\rho_{H}:T\to \ho$ be the moment maps for
the actions on $T$, and
$\rho_{G}^{H}:H\to \go$ the moment map for the action of $G$ on $H$ by
isomorphisms.  As above, we use notations $L(x)$ and $\pi(h)$ for the
partially defined maps on $T$ induced by $x\in G$ and $h\in H$, and
$\alpha_{x}$ for the isomorphism of $T_{s(x)}$ onto $T_{r(x)}$.  Then
we want these actions to satisfy a similar covariance condition
\eqref{eq:51} even though both sides are not everywhere defined.  Thus, 
written in the usual notation for groupoid actions, we want
\begin{align}
  \label{eq:52}
  x\cdot (h\cdot t)=(x\cdot h)\cdot (x\cdot t)
\end{align}
to hold when both sides are defined.   More precisely, let
\begin{align}
  \label{eq:62}
  L=\set{(x,h,t):\text{$s_{H}(h)=\rho_{H}(t)$
  and $s_{G}(x) =
    \rho_{G}(h\cdot t)$}}
\end{align}
be the subset of $G\times H\times T$ where the left-hand side of
\eqref{eq:52} is defined, and let 
\begin{align}
  \label{eq:63}
  R=\set{(x,h,t):\text{$s_{G}(x)=\rho_{G}^{H}(h)$,
  $s_{G}(x)=\rho_{G}(t)$,
    and $s_{H}(x\cdot h)=\rho_{H}(x\cdot t)$}}
\end{align}
be the subset where the right-hand side is defined.
Then we require that $L=R$,
and that \eqref{eq:52} holds on this common set.

While this might seem overly difficult to establish,
the following observation greatly simplifies
checking the validity of these requirements.

\begin{lemma}
  \label{lem-john1}  Suppose that the $G$- and $H$-actions on $T$ are
  compatible in that the diagram
  \begin{equation}
    \label{eq:64}
    \begin{tikzcd}[column sep=2cm]
      &T\arrow[dl,"\rho_{G}",swap] \arrow[dr,"\rho_{H}"] \\
      \go && \ho \arrow[ll,"\rho_{G}^{H}"]
    \end{tikzcd}
  \end{equation}
  commutes and $\rho_{H}$ is $G$-equivariant in that
  \begin{align}
    \label{eq:65} \rho_{H}(x\cdot t)=x\cdot \rho_{H}(t)\quad\text{if
    $s_{G}(x)=\rho_{G}(t)$.} 
  \end{align}
Then the sets $L$ and $R$ defined above coincide and the left-hand
side of \eqref{eq:52} is defined precisely when the right-hand side
is.  Moreover, if $x\cdot h$ and $h\cdot t$ are both defined, then
$(x,h,t)\in L=R$.   Conversely, if $R$ and $L$ coincide, then
$\rho_{G}(t)=\rho_{G}^{H}\bigl(\rho_{H}(t)\bigr)$ and $\rho_{H}(x\cdot
t)=x\cdot\rho_{H}(t)$ for all $(x,h,t)\in L=R$.
\end{lemma}
\begin{proof}
  Suppose that $L=R$ and $(x,h,t)\in L=R$. 
Since $G$ acts by automorphisms,
$\rho_{G}^{H}(s_{H}(h))=\rho_{G}^{H}(h)$, and hence
\[\rho_{G}^{H}\bigl(\rho_{H}(t)\bigr)=\phi_{G}^{H}(s_{H}(h))=\phi_{G}^{H}(h)
  =s_{G}(x) =\rho_{G}(t).\]

On the other hand,
\begin{align}
  \label{eq:60}
  \rho_{H}(x\cdot t)=s_{H}(x\cdot h)=x\cdot s_{H}(h)=x\cdot \rho_{H}(t).
\end{align}

For the other direction, now assume $\rho_{G}=\rho_{G}^{H}\circ
\rho_{H}$ and that $\rho_{H}$ is $G$-equivariant.

Assume $(x,h,t)\in L$.  Then, since $G$ acts by isomorphisms, 
\begin{align}
  \label{eq:66}
  s_{G}(x)&=\rho_{G}(h\cdot t) = \rho_{G}^{H}\bigl(\rho_{H}(h\cdot
            t)\bigr) \\
&= \rho_{G}^{H}(r_{H}(h)) =\rho_{G}^{H}(h).  
\end{align}
Then
\begin{align}
  \label{eq:67}
  s_{G}(x)&= \rho_{G}^{H}(h) = \rho_{G}^{H}\bigl(s_{H}(h)\bigr) \\
&= \rho_{G}^{H}\bigl(\rho_{H}(t)\bigr) = \rho_{G}(t).
\end{align}
Lastly,
\begin{align}
  \label{eq:68}
  s_{H}(x\cdot h)
&= x\cdot s_{H}(h) = x\cdot \rho_{H}(t) = \rho_{H}(x\cdot t).
\end{align}
 Therefore $(x,h,t)\in R$.

Now suppose that $(x,h,t)\in R$.   Then
\begin{align}
  \label{eq:69}
  x\cdot s_{H}(h)
&= s_{H}(x\cdot h) = \rho_{H}(x\cdot t) = x\cdot \rho_{H}(t).
\end{align}
But then $s_{H}(h)=\rho_{H}(t)$.
But then we also have
\begin{align}
  \label{eq:70}
  s_{G}(x) 
&= \rho_{G}(t) =\rho_{G}^{H}\bigl(\rho_{H}(t)\bigr) = \rho_{G}^{H}
              \bigl(s_{H}(h) \bigr)\\
&= \rho_{G}^{H}\bigl(r_{H}(h)\bigr)= \rho_{G}^{H}\bigl(\rho_{H}(h\cdot 
                                        t\bigr) \\
&=\rho_{G}(h\cdot t).
\end{align}
Thus $(x,h,t)\in L$.  Thus $L=R$ as claimed.

Now suppose that $x\cdot h$ and $h\cdot t$ are defined.   That means
$s_{G}(x)=\rho_{G}^{H}(h) $ and $s_{H}(h)=\rho_{H}(t)$.   Then
\begin{align}
  \label{eq:71}
  s_{H}(x\cdot h)=x\cdot s_{H}(h) = x\cdot \rho_{H}(t)=\rho_{H}(x\cdot t).
\end{align}
Therefore $(x,h,t)\in R$ and we already know $L=R$.
\end{proof}

Lemma~\ref{lem-john1} and the preceding discussion motivate the
following definition.

\begin{defn}\label{def:covariant}
  Let $G$ act on $H$ by isomorphisms.  We define
  actions of $G$ and $H$ on $T$ to be \emph{covariant} if the moment
  maps commute as in 
  \eqref{eq:64}, if $\rho_{H}$ is $G$-equivariant as in \eqref{eq:65},
  and if the covariance condition
  \begin{equation}\label{cov}
    x\cdot (h\cdot t)=(x\cdot h)\cdot (x\cdot t)
  \end{equation}
  holds whenever $x\cdot h$ and $h\cdot t$ are defined.
\end{defn}

\begin{example}\label{ex-auto-coherent}
  Suppose that $G$ acts on $H$ by isomorphisms.  We can also let $H$
  act on itself by left translation---so that $H$ also plays the role
  of $T$ in \defnref{def:covariant}.  Then $\rho_{H}=r_{H}$ and
  $\rho_{G}^{H}=\rho_{G}$.  The actions 
    are covariant because $G$ acts by isomorphisms.  
\end{example}

Covariant actions allow us to build actions of semidirect products.

\begin{lem}\label{covariant}
  Let $G$ and $H$ be locally compact groupoids, let $T$ be a locally
  compact Hausdorff space, let $G$ act on $H$ by isomorphisms, and let
  $G$ and $H$ act covariantly on $T$.  Then the semidirect product
  groupoid $\hsdg$ acts on $T$ by
  \begin{equation}\label{semi act}
    (h,x)\cdot t=h\cdot (x\cdot t)\quad\text{if $s(x)=\rho_G(t)$ and
      $s(h)=x\cdot \rho_H(t)$,} 
  \end{equation}
  where $\rho_G$ and $\rho_H$ are the moment maps for the actions of
  $G$ and $H$ on $T$, respectively.  The moment map for the action
    \eqref{semi act} is
  $\rho(t) = \bigl(\rho_{H}(t),\rho_{G}(t)\bigr)$. If $\rho_{H}$ is
  open, then so is~$\rho$.
\end{lem}

\begin{proof}
  Recall that
  \[
    \hsdg\units=\set{(v,u)\in H\units\times
    G\units:\rho_G^H(v)=u},
  \]
  where $\rho_G^H:H\to G\units$ is the moment map for the action of
  $G$ on $H$.  The map $\rho:T\to \hsdg\units$ defined by
  \[
    \rho(t)=(\rho_H(t),\rho_G(t))
  \]
  is clearly continuous.  If $\rho_{H}$ is open, we show that $\rho$
  is as well using Fell's Criterion.  Suppose that
  $(v_{n},u_{n})\to (v,u)=\rho(t)$.  Then $v_{n}\to v=\rho_{H}(t)$,
  and we can pass to subnet, relabel, and assume there are
  $t_{n}\in T$ such that $t_{n}\to t$ and $\rho_{H}(t_{n})=v_{n}$.
  Since $(v_{n},u_{n})\in \hsdg\units$, we have
  $\rho_{G}^{H}(v_{n})=u_{n}$.  Then since the actions are compatible,
  $\rho_{G}(t_{n})=\rho_{G}^{H}\bigl(\rho_{H}(t_{n})\bigr) =
  \rho_{G}^{H}(v_{n})=u_{n}$.  It now follows that $\rho$ is open by
Fell's Criterion.
  Since
  \[
    s(h,x)=\bigl(x\inv\cdot
    s_{H}(h),s_{G}(x)\bigr)\righttext{for}(h,x)\in S(H,G),
  \]
  we have $s(h,x)=\rho(t)$ if and only if
  \[
    s_{G}(x)=\rho_G(t) \midtext{and} s_{H}(h)=x\cdot \rho_H(t).
  \]
  Since we are assuming
    $x\cdot \rho_{H}(t)=\rho_{H}(x\cdot t)$, the operation in
  \eqref{semi act} is well-defined.

  Moreover, we can easily check that $\rho(t)\cdot t=t$.  To check
  the action condition, we first need to
  see that $\rho$ is equivariant.  But
  \begin{align}
    \label{eq:42}
    \rho\bigl((h,x)\cdot t\bigr)
    &= \rho\bigl(h\cdot (x\cdot t)\bigr) \\
    &= \bigl(\rho_{H}(h\cdot (x\cdot t)),\rho_{G}(h\cdot (x\cdot
      t))\bigr) \\
    &= \bigl(r_{H}(h),\rho_{G}(h\cdot (x\cdot t))\bigr) \\
    &=\bigl(r_{H}(h),\rho_{G}^{H}(\rho_{H}(h\cdot (x\cdot t))\bigr) \\
    &=\bigl(r_{H}(h),\rho_{G}^{H}(r_{H}(h))\bigr). \\
    \intertext{Since $G$ acts by isomorphisms on $H$, this is }
    &= \bigl(r_{H}(h),\rho_{G}^{H}(h)\bigr) \\
    \intertext{so that, since $(h,x)\in\hsdg$, this is}
    &= \bigl(r_{H}(h),r_{G}(x)\bigr) \\
    &=r(h,x)
  \end{align}
  as required.  Then, on the one hand, if
  $\bigl((h,x),(h',y)\bigr) \in \sd HG\pairs$, then
  \begin{align}
    \label{eq:43}
    (h,x)\cdot \bigl((h',y)\cdot t\bigr)
    &= (h,x)\cdot \bigl(h'\cdot (y\cdot t)\bigr) \\
    &= h\cdot \bigl(x\cdot \bigl(h'\cdot (y\cdot t)\bigr)\bigr).
  \end{align}
  On the other hand,
  \begin{align}
    \label{eq:44}
    \bigl((h,x)(h',y)\bigr)\cdot t
    &= \bigl(h(x\cdot h'),xy\bigr) \cdot t = \bigl(h(x\cdot h')\bigr) \cdot
      \bigl((xy)\cdot t\bigr) \\
    &= h\cdot \bigl[(x\cdot h') \cdot \bigl(x\cdot (y\cdot t)\bigr)\bigr], \\
    \intertext{which, since the actions are covariant, is}
    &= h\cdot \bigl[x\cdot \bigl(h'\cdot (y\cdot t)\bigr)\bigr]
  \end{align}
  as required.

  The continuity is routine: if $(h_i,x_i)\to (h,x)$ in $\hsdg$ and
  $t_i\to t$ in $T$, then
  \[
    (h_i,x_i)\cdot t_i=h_i\cdot (x_i\cdot t_i)\to h\cdot (x\cdot
    t)=(h,x)\cdot t,
  \]
  because $x_i\cdot t_i\to x\cdot t$.
\end{proof}

\begin{cor}\label{semidirect translate}
  Let $G$ and $H$ be locally compact groupoids, let $G$ act on $H$ by
  isomorphisms, and let $H$ act on itself as a space by left
  translation.  Then the semidirect product groupoid $\hsdg$ acts on
  $H$ as a space by
  \[
    (h,x)\cdot k=h(x\cdot k)
  \]
  whenever $k\in H_{s(x)}$ and $s(h) =x\cdot r(k)$.  Moreover, if $G$
  acts on $H$ principally, then so does $\hsdg$.  Furthermore, the
  moment map $\rho'(k)=(r(k),\rho(k))$ for $\hsdg\actleft H$ is open
  provided the range map on $H$ is open.
\end{cor}

\begin{proof}
  In view of Example~\ref{ex-auto-coherent}, the actions of $G$ and
  $H$ on $H$ are covariant since $G$ acts on $H$ by isomorphisms.
  Thus $\hsdg$ acts on $H$ by \lemref{covariant}.

  Suppose that $G$ acts principally on $H$. Then
  by \thmref{classify}, we may replace $H$ by an action groupoid
  $\ac KT$, where $K$ is a locally compact groupoid acting on a
  locally compact Hausdorff space $T$, the moment map
  $\rho_{K}:T\to K\units$ is a principal $G$-bundle, the actions of
  $G$ and $K$ on $T$ commute, and (by \lemref{G on pullback}) the
  action $G\actleft \ac KT$	 is given by $x\cdot (k,t)=(k,x\cdot t)$.
  Note that $\sd{\ac KT}G$ can be described, with a mild abuse of
    notation, as
    \[
      \set{(k,t,x)\in K\times T\times G: \text{$\rho_K(t)=s(k)$ and
          $\rho_G(t)=r(x)$}}. 
    \]
    Moreover, $(k,t,x)$ is a unit if and only if $k$ and $x$ are.  In
  particular, $(k,t,x)$ acts on $(l,u)$ if $\rho_{G}(u)=s_{G}(x)$ and
  $t=l\cdot (x\cdot u)$, in which case
  $(k,t,x)\cdot (l,u)=(kl,x\cdot u)$.

  We want to see that the action of $\sd{\ac KT}G$ on $\ac KT$ is
  principal.  First we verify that this action is free: suppose
  $(k,t,x)\cdot (l,u)=(l,u)$.  Then
  \[
    (kl,x\cdot u)=(l,u)
  \]
  by definition of the action.  Thus $k\in K\units$, and
  $x\in G\units$ since $G$ acts freely on $T$.  Thus $(k,t,x)$ is
    a unit.

  Now we check that the action is proper.  For this, we use \ref{PA4}
  of Proposition~\ref{prop-proper-equiv}.  Suppose that
  $(l_{i},u_{i})\to (l,u)$ and
  $(k_{i},l_{i}\cdot (x_{i}\cdot u_{i}),x_{i})\cdot (l_{i},u_{i}) \to
  (l',u')$.  Then $(k_{i}l_{i},x_{i}\cdot u_{i})\to (l',u')$.  Since
  the $G$-action on $T$ is principal, we can assume that $x_{i}\to x$
  in $G$.  Since $l_{i}\to l$, and $k_{i}l_{i}\to l'$, we must have
  $k_{i}\to k=l'l^{-1}$.  Thus
  $\set{(k_{i}, l_{i}\cdot (x_{i}\cdot u_{i}),x_{i})}$ converges and
  the action is proper.
\end{proof}

Note that in \corref{semidirect translate}, when we replace $H$ by
$\ac KT$,
the range and source maps are
\begin{align*}
  r(k,t,x)&=(r(k),k\cdot t,r(x))
  \\
  s(k,t,x)&=(s(k),x\inv\cdot t,s(x)),
\end{align*}
and the moment map for the action $\sd{\ac KT}G\actleft \ac KT$ is
given by
\[
  \rho'(k,t)=(r(k),k\cdot t,\rho_G(t)).
\]
Moreover, the coordinate projection
\[
  \pi_2:\sd{\ac KT}G\units\to T
\]
is a homeomorphism.

\subsection*{Semidirect product actions on Fell bundles}

Let $p:\AA\to H$ be a Fell bundle and assume that $G$ acts on $\AA$ by
isomorphisms.  Let $q:\EE\to T$ be an upper semicontinuous Banach
bundle on which $G$ acts.  We assume that the actions
$e\mapsto x\cdot e$ on the fibres are Banach space
isomorphisms.
  We
also want $\AA$ to act on $\EE$ as in \defnref{def:fell-bundle-equiv}.

\begin{defn}
  We say that the actions of $G$ and $\AA$ on $\EE$ are
  \emph{covariant} if the underlying actions of $G$ and $H$ on $T$ are
  covariant and
  \begin{equation}
    \label{eq:45}
    x\cdot (a\cdot e)=(x\cdot a)\cdot (x\cdot e)
    \quad\text{whenever $(x, a)\in G*\AA$ and $(a,e)\in \AA*\EE$.}
  \end{equation}
\end{defn}
Note that it follows as in Lemma~\ref{lem-john1} that
if $(x, a)\in G*\AA$ and $(a,e)\in \AA*\EE$, then both sides of
\eqref{eq:45} are defined.

\begin{cor}\label{Fell Semi Action}
  With the above notation, if the actions of $G$ and $\AA$ on $\EE$
  are covariant, then the semidirect-product Fell bundle $\sd \AA G$
  acts on $\EE$ by
  \[
    (a,x)\cdot e=a\cdot (x\cdot e)
  \]
  whenever the right-hand side makes sense.
\end{cor}

\begin{proof}
  First of all, to see that the pairing is
  continuous, let $((a_i,x_i),e_i))\to ((a,x),e)$ in $\sd \AA G*\EE$.
  Then $a_i\to a$, $x_i\to x$, and $e_i\to e$, so
  $x_i\cdot e_i\to x\cdot e$, and then
  $a_i\cdot (x_i\cdot e_i)\to a\cdot (x\cdot e)$, so
  $(a_i,x_i)\cdot e_i\to (a,x)\cdot e$.  Routine exercises in the
  definitions show that the pairing covers the action
  $\ac HG\actleft T$, $((a,x)(b,y))\cdot e=(a,x)\cdot ((b,y)\cdot e)$,
  and $\|(a,x)\cdot e\|\le \|(a,x)\|\|e\|$.
\end{proof}

\section{Quotient action}

\begin{prop}\label{prop-quot-act}
  Let $G$ and $H$ be locally compact Hausdorff groupoids, and let $G$
  act on $H$ principally and by isomorphisms.  Then there is a
  principal right action of the quotient groupoid $G\under H$ on $H$,
  given by
  \begin{equation}\label{eq-quot-act}
    h\cdot(G\cdot k)=hk'
  \end{equation}
  whenever $s(h)\in G\cdot r(k)$, where $k'$ is the unique
  representative of $G\cdot k$ with $r(k')=s(h)$.  Moreover, the
  associated moment map $H\to (G\under H)\units$ is open.
\end{prop}

\begin{proof}
  Note first of all that \eqref{eq-quot-act} makes sense because $G$
  acts freely by isomorphisms.  By \thmref{classify} we can replace
  $H$ by $\ac KT$, where $K=G\under H$, $T=H\units$, $G$ acts on
  $\ac KT$ by $x\cdot (k,t)=(k,x\cdot t)$, and the actions of $G$ and
  $K$ on $T$ commute.  Then by \lemref{G on pullback}, $K$ acts
  principally on the right of $\ac KT$ by
  \[
    (k,t)\cdot \ell=(k\ell,\ell\inv\cdot t).
  \]
  Moreover, the isomorphism $G\under(\ac KT)\simeq K$ of \lemref{G on
    pullback} transforms the action $(\ac KT)\actright K$ into the
  action of $G\under(\ac KT)$ described in the statement of the
  proposition, and hence this action is principal.

  Furthermore, under the isomorphism the moment map
  $\sigma:\ac KT\to K\units$ is
  \[
    (k,t)\mapsto s(k),
  \]
  which is open because it is a composition of two open maps.
\end{proof}

\begin{cor}
  \label{cor-quot-act-bundle} 
  Let $G$ and $H$ be locally compact groupoids, and let $p:\AA\to H$
  be a Fell bundle.  Suppose that $G$ act principally and by
  isomorphisms on $\AA$.  Then the orbit Fell bundle $\gua$ acts on
  the right of the Banach bundle $\AA$ by
  \begin{align}
    \label{eq:46}
    a\cdot (G\cdot b)=ab'
  \end{align}
  whenever $s(a)\in G\cdot r(b)$, and where $b'$ is the unique
  representative of $G\cdot b$ with $s(a)=r(b)$.
\end{cor}
\begin{proof}
  The proof follows the lines of \propref{prop-quot-act}.  Note first
  that \eqref{eq:46} makes sense because the range map
  $r:\AA\to H\units$ is equivariant for free $G$-actions.  By
  \thmref{classify Fell} we can replace $\AA$ by $\ac \BB T$, and by
  \lemref{G on pullback Fell}, $\BB$ acts
  on the right of $\ac \BB T$ by
  \[
    (b,t)\cdot c=(bc,p(c)\cdot t).
  \]
  Moreover, the isomorphism $G\under(\ac \BB T)\simeq \BB$ of
  \lemref{G on pullback Fell} transforms the action
  $(\ac \BB T)\actright \BB$ into the action of $G\under(\ac \BB T)$
  described in the statement of the proposition.
\end{proof}

\begin{prop}
  \label{prop-basic-equiv} 
  Let $G$ and $H$ be locally compact Hausdorff groupoids, and let $G$
  act on $H$ principally and by isomorphisms.  Then $H$ is a
  $\hsdg-\guh$-equivalence.
\end{prop}

\begin{remark}
  In view of 
  \cite{hkqwsemi}*{Example~6.2}, if $H$ is a principal
  $G$-space, then we recover the well-known result that  in this case 
  the action groupoid $G\rtimes H$ is equivalent to the orbit space
  $\guh$ (cf., \cite{wil:toolkit}*{Example~2.35}).
\end{remark}

\begin{proof}
  Note that by 
  \cite{hkqwsemi}*{Corollary~8.5} 
  and
  \propref{prop-quot-act}, $H$ has principal left and right actions by
  $\hsdg$ and $G\under H$, respectively.  As in the proof of
  \propref{prop-quot-act}, we use \thmref{classify} to replace $H$ by
  $\ac KT$, where $K=G\under H$, $T=H\units$, $G$ acts on $\ac KT$ by
  $x\cdot (k,t)=(k,x\cdot t)$, and the actions of $G$ and $K$ on $T$
  commute, and moreover $K$ acts on the right of $\ac KT$ by
  $(k,t)\cdot \ell=(k\ell,\ell\inv\cdot t)$.  Since $\sd{\ac KT}G$
  acts on $\ac KT$ by
  \[
    (k,t,x)\cdot (\ell,s) =(k,t)\bigl(x\cdot (\ell,s)\bigr)
    =(k,t)(\ell,x\cdot s) =(k\ell,x\cdot s)
  \]
  if $t=\ell\cdot x\cdot s$, and since the actions of $G$ and $K$ on
  $T$ commute, it is clear from the formulas that the actions of
  $\sd{\ac KT}G$ and $K$ on $\ac KT$ commute.

  If $\rho'$ and $\sigma$ are the moment maps for the
    $\sd{\ac KT}G$- and $K$-actions, respectively, then they are open
    by 
  \cite{hkqwsemi}*{Corollary~8.5} 
    and \propref{prop-quot-act}.
    Since $K$ and $\ac KT$ have open range and source maps by
    Proposition~\ref{prop-groupoid}, the quotient maps for the 
    left and right actions are also open.  Hence it suffices to show that
    the induced maps $\jqphi$ and $\jqpsi$
    making the diagram
  \begin{equation}
    \label{eq:23}
    \begin{tikzcd}[column sep=2cm, row sep = 1cm]
    \sd{\ac KT}G\units & \ac KT \arrow[r,"\sigma"] \arrow[l,"\rho'",swap]
    \arrow[dl,"Q"] \arrow[dr,"Q",swap] & K\units \\
    (\ac KT)/K \arrow[u,"\jqphi",dashed]&& \sd{\ac KT}G\under (\ac KT)
    \arrow[u,"\jqpsi",dashed,swap] 
    \end{tikzcd}
  \end{equation}
  commute
  are bijections, where the $Q$'s are the appropriate quotient maps.
  (See \cite{wil:toolkit}*{Remark~2.31}.)  Thus it suffices to show
  that for all $(k,t),(\ell,s)\in \ac KT$, if
  $\rho'(k,t)=\rho'(\ell,s)$ then $(\ell,s)\in (k,t)\cdot K$, and if
  $\sigma(k,t)=\sigma(\ell,s)$ then
  $(\ell,s)\in \sd{\ac KT}G\cdot (k,t)$.

  For the first, we have $r(k)=r(\ell)$, so $\ell=kk\inv \ell$, and
  $k\cdot t=\ell\cdot s$, so $s=\ell\inv k\cdot t$, and hence
  $(\ell,s)=(k,t)\cdot k\inv \ell$.

  For the second, we have $s(k)=s(\ell)$, so $\ell=\ell k\inv k$.  We 
  also have $\rho_K(t)=\rho_K(s)$.  Since $\rho_{K}$ is the
  restriction of the quotient map to $T=H\units$ for the principal $G$
  action on $H$, we have $s=x\cdot t$ for a unique $x\in G$.  Hence
  $(\ell,s)=(\ell k\inv,r,x)\cdot (k,t)$ with
  $r=k\cdot x\cdot t$.
\end{proof}

\section{Principal Action Groupoids}

In this section we prove our main result, the equivalence theorem,
in two forms: Theorems~\ref{equiv} and \ref{one sided},
since each is appropriate in its own
setting.

\begin{thm}\label{equiv}
Suppose that \lcg s $G$ and $K$ have commuting actions on a \lchs\ $T$
such that the moment map $\rho_K:T\to K\units$ is a principal $G$-bundle.
Suppose further that $\BB\to K$ is a Fell bundle.
Then the action Fell bundle $\ac \BB T$ becomes a
$\sd{\ac \BB T}G-\BB$-equivalence with operations given as follows:
\begin{enumerate}[label=\textup{(\arabic*)}]
  \item\label{it:equivleftmod}
$\sd{\ac \BB T}G$ acts on the left of the Banach bundle $\ac \BB T$ by
\[
(b,t,x)\cdot (c,s)=(bc,x\cdot s)
\righttext{if}s(b)=r(c),t=x\cdot p(c)\cdot s;
\]

  \item\label{it:equivleftip}
the left inner product is given by
\[
\pre *\<(b,t),(c,s)\>=(bc^*,x\cdot p(c)\cdot s,x)
\righttext{if}s(b)=s(c),t=x\cdot s;
\]

  \item\label{it:equivrightmod}
$\BB$ acts on the right of $\ac \BB T$ by
\[
(b,t)\cdot c=(bc,p(c)\inv\cdot t)
\righttext{if}s(b)=r(c);
\]

  \item\label{it:equivrightip}
the right inner product is given by
\[
\<(b,t),(c,s)\>_*=b^*c
\righttext{if}r(b)=r(c).
\]
\end{enumerate}
Consequently,
if the Haar system on the action groupoid $\ac KT$ is $G$-invariant, then
letting $\alpha$ denote the associated action of $G$ on $\cs(\ac KT,\ac \BB T)$
from \cite{hkqwsemi}*{Lemma~7.2},
the crossed product $\cs(\ac K T,\ac \BB T)\rtimes_\alpha G$
is Morita equivalent to the Fell-bundle \cs-algebra $\cs(K,\BB)$.
\end{thm}

\begin{proof}
Since 
the actions of $G$ and $K$ on $T$ commute,
it is clear from the formulas that the actions of
$\sd{\ac \BB T}G$ and $\BB$ on $\ac \BB T$ commute.

We showed in the proof of \propref{prop-basic-equiv} that
$\ac KT$ is a $(\sd{\ac KT}G,K)$-equivalence.

We write
the 
other two bundle projections as
\begin{align*}
p':\ac \BB T\to \ac KT&\midtext{given by}p'(b,t)=(p(b),t)\\
p'':\sd{\ac \BB T}G\to \sd{\ac KT}G&\midtext{given by}
p''(b,t,x)=(p(b),t,x).
\end{align*}

Routine computations verify the conditions of
\cite{mw:fell}*{Section~6}:
\begin{enumerate}[(a),nosep]
\item
$p''\bigl(\pre *\<(b,t),(c,s)\>\bigr)
=\pre *[p'(b,t),p'(c,s)]$
and
\\$p\bigl(\<(b,t),(c,s)\>_*\bigr)=[p'(b,t),p'(c,s)]_*$;

\item
$\pre *\<(b,t),(c,s)\>^*=\pre *\<(c,s),(b,t)\>$
and
$\<(b,t),(c,s)\>_*=\<(c,s),(b,t)\>$;

\item
$\<(b,t,x)\cdot (c,s),(d,u)\>=(b,t,x)\<(c,s),(d,u)\>$
and
\\$\<((b,t),(c,s)\cdot d\>_*=\<(b,t),(c,s)\>_*d$;

\item\label{it:ip}
with the above module actions and inner products,
for each $(k,t)\in \ac KT$,
$(\ac \BB T)_{(k,t)}$
is a
$\sd{\ac \BB T}G_{\rho'(k,t)}-\BB_{\sigma(k,t)}$ \ibm.
\end{enumerate}
Note that
\begin{align*}
(\ac \BB T)_{(k,t)}&=B_k\times \{t\}\\
\sd{\ac \BB T}G_{(k,t,x)}&=B_k\times \{(t,x)\}\\
\rho(b,t)'&=\bigl(r(b),p(b)\cdot t,\rho_G(t)\bigr)\\
\sigma(b,t)&=s(b),
\end{align*}
so~\ref{it:ip}
follows immediately from the Fell-bundle  axiom that
$B_k$ is a $B_{r(k)}-B_{s(k)}$ \ibm.

Now assume that the Haar system on the action groupoid $\ac KT$ is
$G$-invariant. 
Then $\sd {\ac KT}G$ admits a Haar system by
\cite{hkqwsemi}*{Theorem~6.4}.
Thus \cite{mw:fell}*{Theorem 6.4} applies, and we have a
Morita equivalence  between $\cs(\sd {\ac KT}G,\sd {\ac \BB T}G)$ and
$\cs(K,\BB)$.   Then by 
\cite{hkqwsemi}*{Theorem~7.3}, we have
a Morita equivalence
  \begin{equation}\label{one side morita}
    C^*(\ac KT,\ac \BB T)\rtimes_\alpha G\sim C^*(K,\BB),
  \end{equation}
  as required.  
\end{proof}

The following alternative version of \thmref{equiv}
is a groupoid version of \cite{kmqw2}*{Corollary~3.9}.

\begin{thm}\label{one sided}
  Let a \lcg\ $G$ act on a Fell bundle $p:\AA\to H$ principally and by
  isomorphisms.  Then $\AA$ becomes an $\sd\AA G-G\under \AA$-equivalence with operations given as follows:
  \begin{enumerate}[label=\textup{(\arabic*)}]
  \item\label{it:oneleftmod}
  $\sd\AA G$ acts on the left of the Banach bundle $\AA$ by
    \[
      (a,x)\cdot b=a(x\cdot b)\midtext{if}s(a)=x\cdot r(b);
    \]

  \item\label{it:oneleftip}
  the left inner product is given by
    \[ {}_L\<a,b\>=(a(x\cdot b^*),x)\midtext{if}G\cdot s(a)=G\cdot
      s(b),
    \]
    where $x$ is the unique element of $G$ such that
    $s(a)=x\cdot s(b)$;

  \item\label{it:onerightmod}
  $G\under \AA$ acts on the right of $\AA$ by
    \[
      a\cdot (G\cdot b)=ab\midtext{if}s(a)=r(b);
    \]

  \item\label{it:onerightip}
  the right inner product is given by
    \[
      \<a,b\>_R=G\cdot a^*b\midtext{if}r(a)=r(b).
    \]
  \end{enumerate}
  Consequently, if the Haar system on $H$ is $G$-invariant, then
  letting $\alpha$ denote the associated action of $G$ 
  on $C^*(H,\AA)$ from 
  \cite{hkqwsemi}*{Lemma~7.2}, 
  the crossed product
  $\cs(H,\AA)\rtimes_{\alpha}G$ is Morita equivalent to the Fell-bundle \cs-algebra
  $\cs(G\under H,G\under \AA)$. 
\end{thm}

\begin{rem}
  One could prove a symmetric version of \thmref{one sided}, as in
  \cite[Theorem~3.1]{kmqw2}, but we have no applications of such a
  result, so we omit it.
\end{rem}

\begin{proof}
Note that by 
\cite{hkqwsemi}*{Corollary~8.7} and \corref{cor-quot-act-bundle},
$\AA$ has left and right actions by
$\sd \AA G$ and $G\under \AA$,
respectively.
Items 
\ref{it:oneleftmod} and \ref{it:onerightmod}
in the current theorem just reiterate the formulas for convenient reference.

We use \thmref{classify Fell}
to replace $\AA$ by $\ac \BB T$,
where $\BB=G\under \AA$,
$T=H\units$,
$G$ acts on $\ac \BB T$ by $x\cdot (b,t)=(b,x\cdot t)$,
and the actions of $G$ and $K$ on $T$ commute,
and moreover $\BB$ acts on the right of
$\ac \BB T$ by $(b,t)\cdot c=(bc,p(c)\inv\cdot t)$.
=Then by \thmref{equiv},
$\ac \BB T$
is a
$\sd{\ac \BB T}G-\BB$-equivalence.
The isomorphism of \thmref{classify Fell}
transforms the formulas 
\ref{it:equivleftmod}--\ref{it:equivrightip} of \thmref{equiv}
into the formulas \ref{it:oneleftmod}--\ref{it:onerightip}
of the current theorem.

Since $\AA$ is isomorphic to $\BB\rtimes K$ and $\sd{\BB\rtimes K}G$
is isomorphic to $\sd\AA G$, we conclude that $\AA$ is an $\sd \AA
G-\gua$ equivalence as claimed.

Now assume that the Haar system on $H$ is $G$-invariant.
Then $\sd HG$  and $\guh$ admit Haar systems by 
\cite{hkqwsemi}*{Lemma~6.4}
and 
\propref{prop-haar-guh},
respectively. Thus \cite{mw:fell}*{Theorem 6.4} applies, and we have a
Morita equivalence  between $\cs(\sd HG,\sd\AA G)$ and $\cs(G\under
H,G\under \AA)$.   Then by
\cite{hkqwsemi}*{Theorem~7.3}, we have
a Morita equivalence
  \begin{equation}
    C^*(H,\AA)\rtimes_\alpha G\sim C^*(G\under H,G\under \AA).
  \end{equation}
  as required.  
\end{proof}

\section{Stabilization}\label{sec:stabilization}

The
Ionescu-Kumjian-Sims-Williams
stabilization theorem \cite{stabilization}*{Theorem~3.7}
(which was based upon
\cite[Corollary~4.5]{kum:fell} and
\cite[Theorem~15]{muhly:fell})
says roughly that every Fell bundle over a groupoid is equivalent to a semidirect product.
In this section we will apply \thmref{equiv}
to give a new approach to the stabilization theorem.

Start with a Fell bundle $p:\BB\to G$.
Let $G$ act on itself by left translation,
and let $\ac \BB G\to \ac GG$ be the associated
action Fell bundle.
Now define another left action of $G$ on itself by
\[
x\cdot y=yx\inv.
\]
Call left translation the \emph{first action},
and this new action the \emph{second action}.
Then these two actions of $G$ on $G$ commute,
and the moment map
\[
\rho_G=r_G
\]
for the first action
is a principal $G$-bundle for the second action.
Thus \thmref{equiv} applies,
and \thmref{stabilization thm} below shows how
we recover the stabilization theorem \cite{stabilization}*{Theorem~3.7}.

\begin{thm}[\cite{stabilization}]\label{stabilization thm}
With the above notation,
$\ac \BB G$ is an $\sd {\ac \BB G}G-\BB$ equivalence.
\end{thm}

\begin{rem}
We were lead to our approach to the stabilization theorem
by the following idea:
the above equivalence
should be regarded as a groupoid form of crossed-product duality.  To
see why, suppose $G$ is a group. Then by \cite[Theorems~5.1,
7.1]{kmqw1} we have
\[
C^*\bigl(S(G\rtimes G,G),S(\BB\rtimes G,G)\bigr)
\simeq 
\bigl(C^*(G,\BB)\rtimes_\delta
  G\bigr)\rtimes_{\what\delta} G,
\]
where $\delta$ is the coaction of $G$ on $C^*(G,\BB)$ canonically
associated to the Fell bundle $\BB\to G$
as in \cite[Proposition~3.1]{kmqw1}, and $\what\delta$ is the dual action of $G$ on the crossed product $C^*(G,\BB)\rtimes_\delta G$.  Thus in this case
the Morita equivalence
\[
C^*(\ac GG,\ac \BB G)\rtimes_\alpha G\sim C^*(G,\BB)
\]
that we get by applying the
Yamagami-Muhly-Williams Equivalence Theorem
\cite{mw:fell}*{Theorem~6.4}
to the Fell-bundle equivalence of \thmref{stabilization thm}
gives (the Morita-equivalence form of) Katayama's duality theorem
for maximal coactions and full dual crossed products
\cite[Theorem~8.1]{kmqw1}
(see \cite{katayama}*{Theorem~8} for the original version).
\end{rem}



\end{document}